\documentclass[11pt]{article}

\usepackage[english]{babel}
\usepackage{amsmath,amssymb,amsthm}
\usepackage{fullpage}
\usepackage{graphicx}
\usepackage{color}
\usepackage{transparent}
\usepackage[affil-it]{authblk}

\newtheorem{theorem}{Theorem}[section]
\newtheorem{proposition}[theorem]{Proposition}
\newtheorem{lemma}[theorem]{Lemma}
\newtheorem{corollary}[theorem]{Corollary}

\theoremstyle{definition}

\newtheorem*{remark}{Remark}

\newcommand{\Disk}{\mathbb{D}}
\newcommand{\DD}{\mathbb{D}}
\newcommand{\HH}{\mathbb{H}}
\newcommand{\Arg}{\operatorname{Arg}}
\newcommand{\Z}{\mathbb{Z}}
\newcommand{\N}{\mathbb{N}}
\newcommand{\R}{\mathbb{R}}
\newcommand{\C}{\mathbb{C}}
\newcommand{\ZZ}{\mathbb{Z}}

\newcommand{\hm}{\operatorname{hm}}

\newcommand{\mubub}{\mu^{\textrm{bub}}}
\newcommand{\E}{\operatorname{{\bf E}}}

\newcommand{\bd}{\partial}
\newcommand{\rwpath}[3]{#1: #2 {\rightarrow} #3 }
\newcommand{\dist}{\operatorname{dist}}
\newcommand{\prob}[1]{\mathrm{\textbf{P}} \left ( #1 \right)}

\newcommand{\diam}{\operatorname{diam}}
\newcommand{\J}{\mathcal{J}}
\newcommand{\ball}{\mathcal{B}}
\newcommand{\tJ}{\tilde{\mathcal{J}}}
\newcommand{\TT}{\mathbb{T}}

\let \le \leqslant
\let \leq \leqslant

\let \ge \geqslant
\let \geq \geqslant

\newcommand{\ee}{\epsilon}
\renewcommand{\Re}{\operatorname{Re}}
\renewcommand{\Im}{\operatorname{Im}}

\newcommand{\looper}{\mathcal{J}}

\newcommand{\Prob} {{\bf P}}

\newcommand{\p}{\partial}
\newcommand{{\pe}}  {\partial_e}
\newcommand {{\lodd}} {{\mathcal J}}
\newcommand{{\inrad}} {{\rm inrad}}

\newcommand {{\cent}} {{w_0}}
\newcommand {{\eb}}  {{\bf e}}
\newcommand {{\dsets}} {{\mathcal A}}
\newcommand {{\dsquare} }  {{\mathcal S}}
\newcommand {{\zipper}}  {{\gamma}}
\def \paths {{\mathcal{W}}}
\def \saws {{\mathcal{W}_{\text{SAW}}}}

\newcommand{\squ}{{U}}
\newcommand{\slit}{{U}^-}
\newcommand{\bigo}[1]{O\left(#1\right)}
\newcommand{\eps}{\epsilon}

\newcommand{\ze}{a}
\newcommand{\Half}{\HH}

\newcommand{\pp}{\mathbf{P}}

\begin{document}

\title{Scaling limit of the loop-erased random walk Green's function}
\author{Christian Bene\v{s}
\thanks{cbenes@brooklyn.cuny.edu}}
\affil{Brooklyn College of the City University of New York}
\author{Gregory F. Lawler%
\thanks{lawler@math.chicago.edu}}
\affil{University of Chicago}
\author{Fredrik Viklund%
\thanks{fredrik.viklund@math.kth.se}}
\affil{KTH Royal Institute of Technology and Uppsala University}

\maketitle

\begin{abstract}
We consider loop-erased random walk (LERW) running between two boundary points of a square grid approximation of a planar simply connected domain. The LERW Green's function is the probability that the LERW passes through a given edge in the domain. We prove that this probability, multiplied by the inverse mesh size to the power $3/4$, converges in the lattice size scaling limit to (a constant times) an explicit conformally covariant quantity which coincides with the SLE$_2$ Green's function.

The proof does not use SLE techniques and is based on a combinatorial identity which reduces the problem to obtaining sharp asymptotics for two quantities: the loop measure of random walk loops of odd winding number about a branch point near the marked edge and a ``spinor'' observable for random walk started from one of the vertices of the marked edge.  

\end{abstract}

\tableofcontents

\section{Introduction and outline of proof}
\label{intro}
\subsection{Introduction}
In this paper we consider loop-erased random walk, LERW, on a square grid. This measure on
self-avoiding paths is obtained by running a simple random walk and successively erasing loops as they form. We  work with a chordal version in a small mesh lattice approximation of a simply connected domain: given two boundary vertices, we chronologically erase the loops of a random walk started at one of the vertices conditioned to take its first step into the domain (along a prescribed edge) and then exit at the other vertex (along a prescribed edge). By linear interpolation this gives a random continuous curve -- the LERW path. It is known that the LERW path has a conformally invariant scaling limit in the sense that it converges in law as a curve \emph{up to reparameterization} to the chordal SLE$_2$ path as the mesh size goes to zero. For details, see \cite{LSW}. We will not use any results about SLE in this paper.

The main theorem of this paper is a different conformal invariance result which does not follow from the convergence of LERW to SLE. We are interested in the probability that the LERW passes through a given edge of a grid approximation of a simply connected domain $D$ and we call this probability the LERW (edge) \emph{Green's function} in $D$. We show that  for edges away from the boundary, this probability, when normalized by the inverse mesh size to the power $3/4$, converges as the mesh size gets smaller to an explicit (up to an unknown lattice-dependent constant) conformally covariant function which coincides with the SLE$_2$ Green's function, $G_D(z; a,b)$. This function is defined as the limit as $\eps \to 0$ of $\ee^{-3/4}$ times the probability that the chordal SLE$_2$ path in $D$ between $a\in\p D$ and $b\in\p D$ visits the ball of radius $\eps$ around $z$. As is shown in \cite{lawler_werness}, a formula for $G_D$ can be written using a covariance rule and the fact that $G_{\DD}(0, e^{2i \theta_a}, e^{2i\theta_b})$ equals $|\sin^3( \theta_a - \theta_b)|$ up to a constant. Several related results have been obtained previously, see below for further discussion.

Let us be more precise. Let $D$ be a simply connected bounded Jordan domain containing $0$.   Write $r_D$ for the conformal radius of $D$ seen from $0$. Let $D_n \subset D$ be an approximating simply connected domain obtained by taking a largest union of squares of side-length $1/n$ centered at vertices of $n^{-1} \ZZ^2$ (see Section~\ref{sect:notation} for details.) Given suitable boundary points $a_n, b_n \in \p D_n$ tending to $a,b$ as $n$ tends to $\infty$, we let $\eta_n$ be a LERW in $D_n$ from $a_n$ to $b_n$ (these points are chosen so that there is a unique edge of $n^{-1}\ZZ^2$ which contains them) and write $e=e_n$ for the edge $[0,1/n]$. Our main result may then be stated as follows:
\begin{theorem}\label{thm:main-intro}
There exists  $0< c_0< \infty$  such
that for all  $D,a,b,$ as above there exists a sequence of approximating
domains $D_n \uparrow D$ with boundary points $a_n \rightarrow a$,
$b_n \rightarrow b$ such that
\[  \lim_{n \rightarrow \infty }
 c_0\,  n^{3/4}\,  \prob{e \subset \eta_n} = r_D^{-3/4}\,\sin^3\left( \pi \hm_D\left(0, (ab) \right) \right),\]
where $r_D$ is the conformal radius of $D$ from $0$, $\hm$ denotes harmonic measure, and $(ab) \subset \p D$ is either
of the subarcs from $a$ to $b$.
\end{theorem}  
The convergence of the domains $D_n \subset D$ is in the Carath\'eodory sense.
We do not determine the value of the lattice dependent constant
$c_0$.  We do give   bounds on the rate
of convergence, but it will be easier to describe them in terms of the discrete result of Theorem \ref{maintheorem}. There are two sources of error. For the discrete approximation $D_n$ there is an error in the LERW probability compared to the SLE$_2$ Green's function for $D_n$; we
give a uniform bound on this error.  There is also an error coming from the approximation
of $D$ by $D_n$; this error depends on the domain $D$. If
$\p D$ is nice, say piecewise analytic (analytic close to $a,b$), the first  error term is larger.

Several authors have studied the LERW Green's function (or ``intensity'' as it is sometimes called) and the closely related growth exponent, that is, the polynomial growth rate exponent as $n \to \infty$ of the expected number of steps of a LERW of diameter $n$. Lawler computed these exponents in dimensions $d \ge 4$ in \cite{greenbook}, where they turn out to be the same as for simple random walk with a logarithmic correction in $d=4$. Kenyon proved that the exponent equals $5/4$ in the planar case and also estimated the asymptotics of the Green's function (up to subpower corrections) for a whole-plane LERW from $0$ to $\infty$ on $\ZZ^2$, see \cite{kenyon}. We will only discuss the planar case in the rest of the paper. Masson gave a different proof of Kenyon's result using the convergence to SLE$_2$ and known results on SLE exponents \cite{LSW-exponents-2} and obtained second moment estimates in collaboration with Barlow \cite{masson}, \cite{masson-barlow}. Kenyon and Wilson computed several exact numeric values for the Green's function of the whole-plane LERW on $\ZZ^2$ in the vicinity of the starting point, see \cite{kenyon-wilson}. In \cite{lawler_lerw_prob} Lawler recently estimated up to constants the decay rate of the Green's function for a chordal LERW in a square domain and the main result of this paper is obtained by refining the arguments of that paper. 
Our use of a branch cut is based on an idea of Kenyon's \cite{kenyon}, as discussed in Section 5.7 of \cite{kenyon-wilson}. 

The present paper is, to our knowledge, the first that treats general simply connected domains and obtains asymptotics. This is critical for the principal application we have in mind, see below. Some of the quantities we consider (and the scaling limit result itself) are related to ones appearing in the analysis of the Ising model, see, e.g., the papers by Hongler and Smirnov and Chelkak and Izyurov \cite{hongler-smirnov}, \cite{CHI_spin_ising}, but we will not use discrete complex analysis techniques here. 

The LERW path is known to converge to the SLE$_2$ path when parameterized by capacity, a parameterization which is natural from the point of view of conformal geometry. An important question is whether the LERW path also converges when parameterized in the natural Euclidean sense so that, roughly speaking, it takes the same number of steps in each unit of time. The conjecture is that one has convergence in law of LERW to SLE$_2$ with a particular parameterization, the \emph{Natural Parameterization}, which can be given as a multiple of the $5/4$-dimensional Minkowski content of the SLE$_2$ curve. See \cite{lawler-rezaei} and the references therein. 
One motivation for studying the problem of the present paper is that we believe it to be a critical step in the proof of this conjecture.    
See also \cite{garban-pete-schramm} for some results for the corresponding question in the case of percolation interfaces converging to SLE$_6$.

The starting point of our proof is a combinatorial identity that factors the LERW Green's function, just as in \cite{lawler_lerw_prob}. We give here a new proof using Fomin's identity \cite{fomin-lerw} which makes more explicit the connection with determinantal formulas. We actually prove a generalization which considers a LERW path containing as a subset a prescribed self-avoiding walk (SAW) away from the boundary. (A given edge is clearly a special case of such a SAW.) From this it follows that there are two factors whose asymptotics need to be understood. The first is the squared exponential of the \emph{random walk loop measure} of loops of odd winding number about a dual vertex next to the marked edge.  We obtain asymptotics by comparing this quantity with the corresponding conformally invariant Brownian loop measure quantity which can be computed explicitly. The second factor can be written in terms of a ``signed'' random walk hitting probability or alternatively as an expectation for a random walk on a branched double cover of the domain (the branch point is the dual vertex mentioned above). 
 After some preliminary reductions the required estimates are proved using coupling techniques that include the KMT strong approximation (see \cite{kmt1}) and results from \cite{KL}, \cite{BJK}. Some of the auxiliary results in this paper may be of independent interest. For instance, we compare various discrete boundary Poisson kernels and Green's functions (near the boundary) with their continuous counterparts in slit square domains and we obtain sharp asymptotics for Beurling-type escape probabilities for random walk started near the slit.    

\subsection*{Acknowledgements} We would like to thank an anonymous referee for carefully reading the manuscript and for providing several useful suggestions. We thank Marcin Lis and Dapeng Zhan for helpful comments on a previous version of the paper.

Bene\v{s} was supported by the Brooklyn College Foundation and PSC-CUNY Award \# 67591-0045.  Lawler was supported by National Science Foundation Grant  DMS-0907143. Viklund was supported by the Simons Foundation, National Science Foundation Grant DMS-1308476, the Swedish Research Council (VR).

\subsection{Notation and set-up}\label{sect:notation}
The proof of Theorem~\ref{thm:main-intro} has three principal building blocks. Although we formulated the theorem for a fixed domain being approximated with a grid of small mesh size we prefer to work with discrete domains in $\ZZ^2$ and let the inner radius from $0$ tend to infinity. 
Let us set some notation.
\begin{itemize}\setlength{\itemsep}{5pt}
\item{
We write the planar integer
lattice $\Z^2$ as $\Z \times i \Z \subset \mathbb{C}$.  Throughout this
paper we fix
\[    w_0 = \frac 12 - \frac i 2 , \]
and note that the dual lattice to $\Z^2$ is $\Z^2 + w_0$. }

\item  A subset
of $A \subset \Z^2$ is called {\em simply connected} if both
$A$ and $\Z^2 \setminus A$ are connected subgraphs of $\Z^2$. Let $\dsets$ denote the set
of simply connected, finite subsets $A$ of
$\Z^2 $ that contain the origin. 

\item{Let  $\overrightarrow{\mathcal{E}}=\left\{[z,w]: \, z,w \in \mathcal{V} \right\}$ be the  directed edge set  of the graph $\ZZ^2 =
 \Z + i \Z$.}
 
 \item{Let  $\p_e A$   denote the {\em edge boundary} of $A$,
that is, the set of ordered pairs $[a_-,a_+]$ of
lattice points with
$a_- \in A, a_+  \in \Z^2 \setminus
A, |a_- - a_+| = 1$. We sometimes write $\p A$ for the set of such $a_+$ and $\overline{A} = A \cup \p A$.  We will use the symbol $a$ both for the point $(a_-+a_+)/2\in \p D_A $ and for the edge $[a_-,a_+]$. It will be clear from context which of the two is meant.}

\item{ 
For each $z \in \Z^2$, let $\dsquare_z$ denote
the closed square region of side length one centered at $z$,
\[                \dsquare_z =\left \{z + (x+iy) \in \C:
   0 \leq |x|,|y| \leq \frac 12 \right\} . \]
Note that the corners of $\dsquare_z$ are on
the dual lattice $\Z^2 + w_0$.}
\item{
If $A \in \dsets$, let  $D_A \subset \C$ be the 
simply connected domain
\[   D_A = {\rm int} \left[ \bigcup_{z \in A} \dsquare_z
  \right] . \]  This is a Jordan domain such that $A \subset D_A$ and $\p D_A$ is a subset of the edge set of the dual lattice $\ZZ^2 +w_0$. Note that (the midpoint of) each such dual edge determines an edge of $\p_e A$; indeed, the midpoint of the dual edge is also the midpoint of a unique edge in $\p_e A$. }

\item{

Let $f = f_A $  denote the unique conformal
  map  $ f: D_A \rightarrow \Disk$ with 
  \[  f (w_0) = 0 , \;\;\;\; f'(w_0) > 0 . \]

  }

 \item{
For $a \in \p D_A$, we define $\theta_a \in [0,\pi)$  by
\[f_A(a) = e^{i 2 \theta_a},\]
which can be defined by extension by continuity, since $D_A$ is a Jordan domain.  
  Note the factor of $2$ in the definition, which is included in order to make later formulas cleaner. }

  \item{
  Let
  \[r_A = r_A(w_0)= f '(w_0)^{-1} \] be the conformal
  radius of $D_A$ with respect to $w_0$. If $r_A(0)$
  denotes the conformal radius from $0$, then
 one can use Koebe's $1/4$ theorem and the distortion theorem (see \cite{Lawler_cip}) to verify that  $r_A(0)=r_A \; [1+O(r_A^{-1})]$.
  }

\item We write \[\omega = [\omega_0,\ldots,\omega_\tau]\]
for nearest neighbor walks in $\Z^2$ and simply call them walks or paths.  We write
$|\omega| = \tau$ for the length of the path and
$p(\omega) = 4^{-|\omega|}$ for the simple random
walk probability of $\omega$. 

\item We write $\oplus$ for concatenation of paths.  That is
to say if
$\omega^1 = [\omega^1_0,\ldots,\omega^1_k],
\omega^2=[\omega^2_0,\ldots,\omega^2_j]$, the concatenation
$\omega^1 \oplus \omega^2$ is defined if $\omega^1_k = \omega^2_0$,
in which case
\[   \omega^1 \oplus \omega^2 =   [\omega^1_0,\ldots,\omega^1_k,
\omega^2_1,\ldots,\omega^2_j].\]

\item{If $a, b$ are distinct elements of $\p_e A$, we let 
 \[   \paths = \paths(A;a, b) \]
be the set of walks 
\[  \omega = [\omega_0,\ldots,\omega_\tau] , \]
with $[ \omega_0,\omega_1] =[a_+,a_-] ,
[\omega_{\tau-1},\omega_\tau] = [b_-,b_+]$,
and 
$\omega_1,
\ldots,\omega_{\tau-1} \in A$. }
\item{We sometimes write
\[
\omega: x \to y,
\]
where $\omega$ is a walk and where $x$ and $y$ can be edges or vertices, to mean that $\omega$ is a walk starting at $x$, ending at $y$.}
\item

For $a,b\in \p_e A$, we write
\[H_{\p A}(a,b) = \sum_{\omega \in \paths 
} p(\omega)\]
for the corresponding (boundary)
Poisson kernel. If $x \in \ZZ^2 \setminus A$ and $\dist(x,A) = 1$, then we will similarly write  $H_{\p A}(x,b)=\sum_{a:a_+ =x} H_{\p A}(a,b)$.

\item If $\omega \in  \paths(A;a, b)$ with $|\omega| = \tau$, 
we will  also
write
$\omega(t), \frac 12 \leq t \leq \tau - \frac 12,$ 
for the continuous path of time duration
$\tau-1$  
that starts at $a $ and
goes to $b $ along the edges of $\omega$
at speed one. Note that $\omega(t) \in D_A$
for $\frac 12 < t < \tau - \frac 12 .$

\item Let $\saws = \saws(A;a,b)$ denote the set of walks 
$\eta \in \paths$ that are self-avoiding walks, that is, such that $\eta(s)\neq \eta(t)$ for $s<t$. Note
that a path $\omega$ is self-avoiding if and only
if $f \circ \omega$ is a simple curve.  

\item  For 
each $\omega \in \paths$ there exists a unique
$\eta = L(\omega) \in \saws $ obtained by chronological
loop-erasing. (But $L^{-1}(\eta)$ may have many elements.) See Section~\ref{sect:preliminaries}. 

\item Let $\saws^+$ (resp., $\saws^-$) denote
the set of $\eta \in \saws$ that include the ordered
edge $[0,1]$ (resp., $[1,0]$) and $\saws^* =
\saws^+ \cup \saws^-$.   Let $\paths^*$ be the set
of $\omega \in \paths$ such that $L(\omega)
\in \saws^*$.  Set
\[   H_{\p A}^*(a,b) := \sum_{\omega \in \paths^* 
} p(\omega). \]
If $e$ is the unordered edge $[0,1]$ and $\eta$ is a LERW from $a$ to $b$ in $A$, then 
we can write
\begin{equation}\label{intro-p-1}
 P(a,b;A) := \prob{e \subset \eta}= \frac{H_{\p A}^*(a,b) } {  H_{\p A} (a,b)   };
\end{equation}
this is the LERW Green's function at the edge $e$.

\end{itemize}

Our main result is a consequence of the following theorem.

\begin{theorem}  \label{maintheorem}
There exist $u > 0$ and $0 < c_0 < \infty$ such
that the following holds.   Suppose $A \in {\mathcal A}$
and 
suppose $a,b
\in \p_e A$ with $\left| \sin( \theta_{a } - \theta_{b} ) \right| \geq r_A^{-u}$.
Then, 
\[    P(a,b;A) 
=   c_0 \, r_A^{-3/4}  |\sin^3\left(\theta_{a } - \theta_{b}\right)|
    \left[ 1 + O\left( 
  r_A^{-u} \, \left| \sin\left( \theta_{a } - \theta_{b}\right)  \right|^{-1}   
 \right)\right] .   \]
\end{theorem}

Let us explain how
to derive Theorem~\ref{thm:main-intro}
from Theorem \ref{maintheorem}. Along the way we will comment on convergence rate bounds. Suppose $D$ is a Jordan domain
containing the origin with $\dist(0,\p D) = 1$.  For each $n$, let
$A_n$ be the largest simply connected subset of $\Z^2$ containing the
origin such that $\overline{D_{A_n}} \subset n D$.  Let $D_n = n^{-1} D_{A_n}$,
$w_n = n^{-1} w_0 = n^{-1}(1/2- i/2)$.
Let $f_{A_n}$ be the corresponding uniformizing conformal map as above and $F_n(z) = f_{A_n}(nz)$.
Then $F_n: D_n \rightarrow \Disk$ with $F_n(w_n) = 0, F_{n}'(w_n) >0$.
Let $F:D \rightarrow \Disk$ be the conformal transformation with $F(0) = 0,
F'(0) >0$.  Since $D$ is a Jordan domain, $F$ extends to a homeomorphism,
$F: \overline D \rightarrow \overline \Disk$. Note that $D_n \subset D$ and for each $n$ we can write
\[
F_n(z) = M_{n} \circ \psi_n \circ F(z), \quad z \in \overline{D_n},
\]
where $\psi_n : F(D_n) \to \Disk$ with $\psi_n(0)=0, \psi_n'(0)>0$ and $M_n(z) = k_n(z-u_n)/(1-\overline{u_n}z)$ is the M\"{o}bius transformation of $\Disk$ taking $u_n=\psi_n \circ F(w_n) = O(1/n)$ to $0$ with $k_n\in \p \Disk$ chosen so that $[M_{n} \circ \psi_n \circ F]'(w_n) > 0$. We have $[\psi_n\circ F]'(w_n) = \psi_n'(0) F'(0)(1+O(1/n))$ and consequently $|k_n-1|=O(1/n)$. It follows that if $|z| > c$ for some constant $c$, then $M_n(z) = z(1+O(1/n))$, where the error depends on $c$.  

If $z \in \p D_n$, let $w$ be a point in $\p D$ such that $|z-w|=\dist(z,\p D)$.
Since $z$ is contained in a closed square of side length $n^{-1}$ that intersects
$\p D$, we have $|z-w| \leq \sqrt 2/n$.  By the Beurling estimate (see, e.g., \cite{Lawler_cip}), we
can see that there is a universal constant $c$ such that $|F(z) - F(w)| \leq c \, n^{-1/2}.$  In other words, 
there exists $c_1$ such that
\begin{equation}\label{annul}         F(\p D_n) \subset \{z \in \overline{\Disk}: |z| \geq 1 - c_1n^{-1/2}\}.\end{equation}
Also, by the Beurling estimate, if $e$ is an edge on the boundary of $A_n$ (and $\diam e = 1$), 
then $\diam F(n^{-1}e) = O(n^{-1/2})$.   Let $a \in \p D$ be arbitrary. We will choose a particular point $a_n \in \p D_{n}$ so that $|F(a) - F_n(a_n)|$ is small. Let $v_n \in \Disk$ be a point on $F(\p D_n)$ with  the same argument as $F(a)$ and with minimal radius. Then as we showed above $|v_n| \ge 1- c_1 n^{-1/2}$ and there exists an edge $e \subset \p D_{A_n}$ (this is an edge of the dual to $\ZZ^2$) such that $v_n \in F(n^{-1} e)$. We take $a_n \in \p D_n$ to be the midpoint of $n^{-1} e$. Note that $na_n$ then determines an element of $\p_e A_n$ by virtue of being its midpoint. We have $|F(a_n)-F(a)| = O(n^{-1/2})$. It is not hard to show that if $c_1$ is as in \eqref{annul} then for $z$ with $|z| \le 1-2 c_1 n^{-1/2}$, 
\[
|\psi_n(z) - z| \le c_2 n^{-1/2} \log n,
\] 
where $c_2$ is universal. See, e.g., Section~3.5 of \cite{Lawler_cip}. Using this and the Beurling estimate we see that $|\psi_n \circ F(a_n) - F(a)| = O(n^{-1/5})$. (We are not attempting to optimize exponents here.) Using the estimate on $M_n$ it follows that \[|F_n(a_n) - F(a)| = O(n^{-1/5}).\]

Also, $r_D = n^{-1} r_{A_n}
[1 + O(n^{-1/2})]$.  Hence, given $a,b \in \p D$, we would like to choose $a_n, b_n \in \p D_n$ to approximate $a,b$, and then apply Theorem~\ref{maintheorem} to $A_n$ with boundary edges determined by $n a_n, n b_n$, to get a uniform error term in Theorem~\ref{thm:main-intro}.
 
Unfortunately, although there is a uniform bound on $|F(a) - F(a_n)|$, there
is no uniform bound on $|a - a_n|$ without additional assumptions on the regularity of $\p D$.  However, since
$|F(a) - F(a_n)| \leq c_2 \, n^{-1/2}$, we certainly have
\[  |a - a_n|  \leq  \delta_n:= \sup\left\{|F^{-1}(z) - F^{-1}(w)|
:  |z-w| \leq c_2 \, n^{-1/2} \right\}.\] 
Since $D$ is a Jordan domain $F^{-1}$ is uniformly continuous
and hence $\delta_n \rightarrow 0$ as $n \rightarrow \infty$ and so $a_n \to a$ and $b_n \to b$ and this is all we need for Theorem~\ref{thm:main-intro} without a convergence rate estimate.

If $\p D$ is, e.g., locally analytic at $a$ and $b$, or more generally, if
the map $F$ is bi-Lipschitz in neighborhoods of $a$ and $b$, then
one can improve these estimates giving $|a-a_n| = O(n^{-1}),
|F(a) - F(a_n)| = O(n^{-1})$. Analogous estimates under weaker conditions on $\p D$ can also be given. The conclusion is that for sufficiently ``nice'' domains the
biggest error term in our result comes from the discrete result,
Theorem \ref{maintheorem}.

\subsection{Outline of proof and an important idea}\label{sect:key}
The first step of the proof of Theorem~\ref{maintheorem} rewrites \eqref{intro-p-1} as a product of three factors which will then be estimated in the remainder of the paper. Before stating the main estimates we will introduce an idea which is further discussed in Section~\ref{sect:candidate}.

Suppose $\omega(t), t \in [0,T],$ is a curve   in $\overline{D_A}$ that avoids
$w_0$. Let $t \mapsto \Theta_t = \arg[f(\omega(t))]$ be a continuous version of the argument. Define
\begin{align*}  J_t&  = \left\lfloor \frac{\Theta_t}{2 \pi} \right\rfloor -
     \left\lfloor \frac{\Theta_0}{2 \pi}\right \rfloor ;\\
       Q_t & = Q(\omega[0,t])   = (-1)^{J_t}. \end{align*}
       Although the argument is only defined up to an integer
multiple of $2 \pi$, 
the value of $J_t$, and hence the value of $Q_t$
are independent of the choice of $\Theta_0$. 
If $\omega$ has time duration $\tau$, we write
$    Q(\omega) = Q_{\tau}$.   Note that if $\omega = \omega_1
 \oplus \omega_2$, then
 \begin{equation}\label{5-31}
      J(\omega) = J(\omega_1)  + J(\omega_2), \;\;\;\;
     Q(\omega) = Q(\omega_1) \, Q(\omega_2) . 
 \end{equation}
   In particular, if $\omega = [\omega_0,\ldots,\omega_\tau]$ is
   a path lying in $A$, then
   \[   Q(\omega) = \prod_{j=1}^\tau Q(e_j) , \;\;\;
     e_j = [\omega_{j-1}, \omega_j] . \]
     Roughly speaking $Q(e) = -1 $ if and only if
   the edge $e$ crosses the branch cut $\beta :=f_A^{-1}([0,1])$.
  We
  note the following:
  \begin{itemize}\setlength{\itemsep}{5pt}
  \item  $Q(\overrightarrow e)$ is a function of the undirected edge $e$.
  \item  If $e = [0,1]$,
  then $Q (e ) = 1$ assuming $r_A$ is sufficiently large. We will assume this throughout the paper.  
  \item 
  if $\ell $  is a loop in $A$, then 
   $Q(\ell) = -1$
if and only if the winding number of $\ell$ about
$w_0$ is odd.
\end{itemize}

We define  \emph{(signed) weights}   by
\[
q(e)  =  p(e) \, Q(e) = \frac{1}{4}Q(e), \quad (e \text{ edge}),
\]  and if $\omega$ is a walk as above,
\[   q(\omega) = p(\omega) \, Q(\omega) =
   \prod_{j=1}^\tau q(e_j) . \]

Let $S_j$ be simple random walk
starting in $A$, and let $\tau = \tau_A=\inf\{k\geq 0: S(k)\not\in A\}$
be the first time that the walk leaves $A$. As a slight abuse of notation, we write $S_{\tau}=
a$ to mean that the walk exits $A$ through
the ordered edge $[a_-,a_+]$. 
If $S_{\tau} = a$, we associate
to the random walk path the continuous path in $\overline{D_A}
$ of time
duration $\tau - \frac 12$ ending at $a \in \p D_A$.

Let
 
\begin{equation}\label{Ia}
I_a=1\{S[1,\tau] \cap \{0,1\} = \emptyset; S_{\tau}=a \}
\end{equation}
be the indicator of the event that $S$ leaves $A$ at the boundary edge $a$ and never visits the points $0,1$ before leaving $A$.  Let
\[
R_A(z,a)=\E^z\left[(-1)^{J(S[0,\tau-\frac 12])} I_a \right]
 = \E^z\left[Q(S[0,\tau-\frac 12]) \,I_a \right],
\] 
\[
\Phi_A(a,b) = \frac{|R_A(0,a) R_A(1,b) - R_A(0,b) R_A(1,a)|}{H_{\p A}(a,b)}.
\]

Let $G^q_A(z,w)$ denote the random walk Green's function in $A$ using the
signed weight $q$, 
\[       G^q_A(z,w)  = \sum_{\stackrel{\omega: z \rightarrow w}{\omega\subset A}}
                  q(\omega) . \]
     Here the sum is over all walks in $A$ from $z$ to $w$. 
 From the definition, we can write 
\[     G^q_A(0,0) = \sum_{j=0}^\infty
        \E^0\left[Q\left(S[0,j] \right) ;S_j = 0 ; j < \tau_A\right], \] 
\[     G^q_{A \setminus \{0\}}(1,1) = \sum_{j=0}^\infty
        \E^1\left[Q\left(S[0,j] \right) ;S_j = 1 ; j < \tau_{A \setminus \{0\}}  
         \right], \]
   We define 
  \[     \bar{q}_A = \frac 14 \, G^q_A(0,0) \,
       G^q_{A \setminus \{0\}}(1,1).\]
   We can also interpret $4 \, \bar q_A$ as the random walk loop
   measure using the weight $q$
    of loops in $A$ that intersect $\{0,1\}$.

We write $\looper_A$ for the set of unrooted random walk loops  $\ell \subset A$ with $Q(\ell)=-1$. (See 
Section~\ref{sect:preliminaries} for precise definitions.) The following is the combinatorial identity central to our proof:
\begin{theorem} \label{thm:det-formula}  Let $A \in 
\dsets$ 
and $a,b \in \p_e A$.
 Then,
\begin{equation}\label{de} P(a,b; A) =  \bar{q}_A \, \exp \left\{  2 m(
\looper_A)\right\} \, \Phi_A(a,b), \end{equation}
where $m$ is the random walk loop measure and  $\looper_A$ is the set of unrooted random walk loops  $\ell \subset A$ with $Q(\ell)=-1$. 
\end{theorem}
\begin{proof}
See Section~\ref{sect:det}.  
\end{proof}

It is not hard (see       \cite[Section 2]{lawler_lerw_prob})
 to see that there exists $\bar q \in (0,\infty)$  and $u > 0$
 such that 
\begin{equation}  \label{sept20.1}
     \bar q_A = \bar q + O(r_A^{-u}) . 
     \end{equation}
To obtain Theorem \ref{maintheorem}, the remaining work is then to  estimate the other two factors on the right-hand side of \eqref{de}. In Section~\ref{sect:loops} we compare the random walk loop measure with the Brownian loop measure to prove
the following.   Our proof does not yield the value of
the lattice-dependent constant $c_1$.  

\begin{theorem}  \label{thm:main-thm-loop-measure}
There exist   $u>0$ and $0<c_1 < \infty$ such that
if $A \in \dsets$,
\[     \exp\{2m(\looper_A)\}
   = c_1 \, r_A^{1/4}  \, \left[1 + O\left(r_A^{-u}
    \right) \right] .\]
\end{theorem}
\begin{proof}
See Section~\ref{sect:loops}. 
\end{proof}
The last factor in \eqref{de} is estimated using the following result, the proof of which is the main technical hurdle of the paper. For $z\in A$ and $b\in \p_e A$, let
\begin{equation}\label{poissondiscrete}
H_{A}(z,b) = \Prob^z(S_{\tau-1}=b_-, S_{\tau}=b_+)
\end{equation}
be the discrete Poisson kernel and
\[   \Lambda_A(z,a) = \frac{R_A(z,a)}{H_{ A}(z,a)}. \]
\begin{theorem} \label{thm:theorem4-new}
There exists $ u > 0$ and  $0 < c_2 < \infty$ such that if $A \in \dsets$,
\begin{equation}  \label{t41}
\Lambda_A(0,a) 
      =   c_2 r_A^{-1/2}\, \left[ \sin  \theta_a   + O\left(r_A^{-u}\right) \right];
    \end{equation} 
    \begin{equation}  \label{t42}
\Lambda_A(1,a) 
      =   c_2 \, r_A^{-1/2} \,\left[ \cos \theta_a + O\left(r_A^{-u}\right) \right] .
    \end{equation} 
\end{theorem}

\begin{proof}
 See Section~\ref{sect:spinor convergence}.  
 \end{proof}
 
\begin{proof}[of Theorem~\ref{maintheorem} assuming Theorems~\ref{thm:det-formula}, \ref{thm:main-thm-loop-measure}, and \ref{thm:theorem4-new}.]
By Theorem~\ref{thm:theorem4-new},
\begin{eqnarray*}
\lefteqn{|\Lambda_A(0,a)\Lambda_A(1,b) - \Lambda_A(1,a)\Lambda_A(0,b)|}
  \hspace{1in}\\
    & = & c_3 r_A^{-1}\left| \sin(\theta_a)\cos(\theta_b) - \cos(\theta_a)\sin(\theta_b) \right| + O\left(r_A^{-1-u}\right) \\
& =  &c_3 r_A^{-1} \left| \sin \left( \theta_a -\theta_b \right)\right| \left[1  + O\left(\left| \sin \left( \theta_a -\theta_b \right)\right|^{-1}\,  r_A^{-u}\right) \right].
\end{eqnarray*} 
We can then use Theorem 1.1 of \cite{KL} which implies that if $|\sin( \theta_a -\theta_b)|
\geq r_A^{-1/20}$, then
\[ H_{  A}(0,a)\, H_A(0,b)
 = \frac{2}{\pi} \, \sin^2(\theta_a - \theta_b)\, H_{\p A}
   (a,b)\, \left[1 + O(r_A^{-u})\right]
    . \]
    But a difference estimate for discrete harmonic functions (see for instance Theorem 1.7.1 in \cite{greenbook}) shows that $H_A(0,a) = H_A(1,a)(1+O(r_A^{-1}))$ and so by combining these estimates we see that
    \begin{align*}
    \Phi_A(a,b) &= |\Lambda_A(0,a)\Lambda_A(1,b) - \Lambda_A(1,a)\Lambda(0,b)| \, \frac{H_A(0,a) \, H_A(0,b)}
        {H_{\p A}(a,b) }\, \left[1 + O(r_A^{-1}) \right] 
      \\
    & =c_4 r_A^{-1} |\sin^3( \theta_a - \theta_b)|\left[ 1 + O\left(\left| \sin \left( \theta_a -\theta_b \right)\right|^{-1}\,  r_A^{-u} 
 \right)\right],
    \end{align*}
    and consequently Theorem~\ref{maintheorem} follows from Theorem~\ref{thm:det-formula} combined  with \eqref{sept20.1},
  Theorem~\ref{thm:main-thm-loop-measure},
    and the last equation.  
\end{proof}

\section{Preliminaries}\label{sect:preliminaries} 
This section sets more notation and collects some background material primarily on loop measures and loop-erased walks.


\subsection{Green's functions and Poisson kernels}
We summarize here some definitions and facts about discrete and continuum Green's functions and Poisson kernels. 
\begin{itemize}\setlength{\itemsep}{5pt}
\item{The (Dirichlet) Green's function (or Green's function for Brownian motion) in a simply connected domain $D \subset \mathbb{C}$, with pole at $w$, is the positive symmetric function $g_D(z,w)$ such that $g_D(z,w) + \log|z-w|$ is harmonic in $D$ and $g_D(z,w)=0$ if $w \in \p D$.}

\item{For a domain $D\subset \C$ if $w\in D$, $z\in\partial D$, and $\bd D$ is locally analytic at $z$, the Poisson kernel $h_D(w,z)$ is the density of harmonic measure with respect to Lebesgue measure and can be given as a normal derivative of $g_D$. In particular, for any piecewise locally analytic arc $F\subset \partial D$, 
$$\Prob^w(B(T)\in F) = \int_F h_D(w,z) \, d|z|.$$}

\item{If $w, z, \in \bd D$ and $\bd D$ is locally analytic at both $w$ and $z$, then it is useful to define the {\em excursion Poisson kernel} 
$$h_{\bd D}(w,z) =\lim_{\eps\to 0^+}\epsilon^{-1} \, h_D(w+\eps \mathbf{n}_w,z),$$
where $\mathbf{n}_w$ is the unit vector normal to $\bd D$ at $w$, pointing into $D$. Note that $h_{\p D}$ can be directly defined as a constant times the repeated normal derivatives in both variables of the Green's function, and we see that it is symmetric and conformally covariant. }

\item{One can define the excursion Poisson kernel at points where $\bd D$ is not locally analytic, such as at the tip of the slit of (smoothly) slit domains. The slit we consider in this paper is the positive real half-line and so the tip is the origin. Applying a conformal map to the unit disk, say, we can see that at such points, the derivative grows exactly like the inverse of the square root of the distance to the tip. So for such slit domains we may define 
\begin{equation*}\label{boundarypoisson}
h_{\bd D}(0,z) =\lim_{\eps\to 0^+}\eps^{-1/2}h_D(-\eps,z).
\end{equation*}}

\item{Similar objects are useful in the discrete setting.
Let $A \subsetneq  \ZZ^2$ be connected. Recall the definition for $w,z\in\bar{A}$ of the random walk Green's
function:
\[
G_A(z,w) = \sum_{ \substack{\omega : z \to w, \\ \omega \subset A}} p(\omega), \quad z,w \in A,
\]
and $G_A(z,w)=0$ if $z\in \p A$ or $w\in\p A$, as well as
the random walk $q$-Green's function of $A$, given in Subsection \ref{sect:key}:
\[
G^q_A(z,w) = \sum_{ \substack{\omega : z \to w, \\ \omega \subset A}} q(\omega), \quad z,w \in A,
\]
and $G^q_A(z,w)=0$ if $z\in \p A$ or $w\in\p A$, where the sums are over walks starting at $z$ and ending at $w$ and staying in $A$.}

\item{Let $A \in \mathcal{A}$ be given.
Another way to define  the Green's function
 is to let $S_j$ be 
  a simple random walk in $A$,  $\tau = \tau_A = \min\{j \geq 0: S_j
   \not \in A \}$, and   
   \[    G_A(z,w) = \E^z \left[ \sum_{j=0}^{\tau-1}
         1\{S_j = w \} \right],\;\;\;\;\; z,w \in A. \]
Using a last-exit decomposition, one can see that the Green's function is related to the Poisson kernel defined in \eqref{poissondiscrete} as follows:
\begin{equation}  \label{lastexit}
               H_A(z,a) = \frac 14 \, G_A(z,a_-) .
\end{equation}
It is known  \cite[Theorem 1.2]{KL} that there exists a  
  lattice-dependent constant $C_0$ and $c < \infty$ such that 
  \[ 
    \left| G_A(0,0) - \frac 2 \pi\,  \log r_A  -  C_0 \right | \leq c \, \frac{\log r_A}{r_A^{1/3}  }. \]
For our purposes, it will suffice to use
\[   G_A(0,0) = \frac 2 \pi \,\log
r_A  + O(1) , \]
  where, as before,  $r_A$ is the conformal radius of $D_A$.}

\item{We will need the following result    
which follows from \cite[(40)--(41)]{KL}.  An
   explicit $u$ can be deduced from these estimates, but
   it is small and not optimal so we will not give it here.

\begin{theorem} \label{MGthm}
There exists $u > 0$ such that if $z \in A$ with $|f(z)|
  \leq 1 - r_A^{-u}$, then
  \[      H_A(z,a) = {H_A(0,\ze)}  \,  \frac{1-|f(z)|^2}
{|f(z)- e^{i2\theta_a}|^2} \, \left[1 + \bigo{r_A^{-u}} \right].\]
and hence, 
\begin{equation}\label{MGthm.eq}
   \frac{h_A(z,a)}{h_A(0,a)} = \frac{H_A(z,a)}{H_A(z,0)}
 \,  \left[1 + \bigo{r_A^{-u}} \right] .
 \end{equation}
Moreover, if $|\theta_a - \theta_b| \geq r_A^{-1/20}$, 

$$   H_{\p A}(a,b) =  \frac{\pi}{2}
     \,\frac{ H_A(0,a) \, H_A(0,b) }
        {\sin^2(\theta_a - \theta_b)}
         \, \left[ 1 + O \left( r_A^{-u}\right) \right]. 
$$

   \end{theorem}}
   \end{itemize}

\subsection{Loops and loop measures}
We will consider oriented rooted \emph{loops} on $\ZZ^2$, that is, walks starting and ending at the same vertex:\[\ell=[\ell_0, \ell_1, \ldots, \ell_{\tau}], \quad \ell_0=\ell_\tau.\] (Unless otherwise stated, we will simply say ``loop''.) The length of a loop, $|\ell|=\tau$, is clearly an even integer. The vertex $\ell_0$ is the \emph{root} of $\ell$. The edge representation of a rooted loop is the sequence of directed edges
\[
e(\ell)=[\overrightarrow{e}_1, \ldots, \overrightarrow{e}_\tau], \quad \overrightarrow{e}_j = [\ell_{j-1}, \ell_j].
\]
Note that each vertex in $\ell$ occurs in an even number of edges in $e(\ell)$.
 
Let $\mathcal{L}_{*}(A)$ be the set of discrete (rooted) loops contained in $A$ and write $\mathcal{L}_*=\mathcal{L}_*({\ZZ^2})$. The {\em  rooted loop measures}, $m_*,m_{*}^q$, associated with $p,q$, respectively are the   measures on rooted loops and are defined by $m _*(\ell) = m^q_*(\ell) = 0$ if $|\ell| = 0$, and otherwise
\[
m_{*}(\ell) = \frac{p(\ell)}{|\ell|}, \quad \ell \in \mathcal{L}_{*}.
\]
\[
m_{*}^q(\ell) = \frac{q(\ell)}{|\ell|}
 =  \frac{p(\ell)\, Q(\ell)}{|\ell|}, \quad \ell \in \mathcal{L}_{*}.
\]
An {\em  unrooted loop} $[\ell]$ is an equivalence class of rooted loops where two rooted loops are equivalent if and only if the edge representation of one can be obtained from that of the other by a cyclic permutation of indices. Clearly the lengths of two equivalent loops are the same, so we can write $|[\ell]|$ for this number. We write $\#[\ell]$ for the number of equivalent rooted loops in $[\ell]$; this is always an integer dividing $|[\ell]|$. Moreover, the weights of all equivalent loops in a given class are the same so we may write $p([\ell])$ and $q([\ell])$.

 We write $\mathcal{L}(A)$ for the set of unrooted loops 
whose representatives are in $\mathcal{L}_*(A)$
  and $\mathcal{L}=\mathcal{L}(\ZZ^2)$.
The {\em  loop measures}, $m,m^q$, are the measure on unrooted loops induced by $m_*,m_{*}^q$:
\[
m ([\ell]) = \sum_{\ell \in [\ell]}m_{*} (\ell) = \frac{\#[\ell]}{|\ell|}p(\ell), \quad [\ell] \in \mathcal{L},
\]  
\[
m^q([\ell]) = \sum_{\ell \in [\ell]}m_{*}^q(\ell) = \frac{\#[\ell]}{|\ell|}q(\ell), \quad [\ell] \in \mathcal{L},
\]  
and where the sums are  over the representatives of $[\ell]$.
If  $\ell$ is a rooted loop, we write $m(\ell), m^q(\ell)$
for  $m([\ell]),m^q([\ell])$.

Suppose that $V \subset A \subset \ZZ^2$, where $A$ is finite. Consider the set of loops contained in $A$ that meet $V$: \[\mathcal{L}(V;A) = \{[\ell] \in \mathcal{L}(A) :  \ell \cap V \neq \emptyset \},\] and  define
$\mathcal{L}_*(V;A)$   similarly using rooted loops.
 Let
 \[
F(V;A)=\exp\left\{  m\left[ \mathcal{L}(V;A) \right] \right\}
 = \exp\left\{  m_*\left[ \mathcal{L}_*(V;A) \right] \right\} ,\]
 \[
F^q(V;A)=\exp\left\{  m^q\left[ \mathcal{L}(V;A) \right] \right\}
 = \exp\left\{  m^q_*\left[ \mathcal{L}_*(V;A) \right] \right\}. \]
It follows from the second equality  that if
  $V=\cup_{i=1}^k V_i$, with the $V_i$ disjoint and
  $A_j = A \setminus (\cup_{i=1}^j V_i)$, then
\begin{equation}\label{lem:FVA}
F^q(V;A) = F^q(V_1;A) F^q(V_2; A_1) \cdots F^q(V_k; A_{k-1}),
\end{equation}
and similarly for $F(V;A)$. 
 In particular, the right-hand side of \eqref{lem:FVA} is independent of the order of the $V_i$ partitioning $V$.
 
 If $V = \{z\}$
is a singleton set, we write just
$\mathcal{L}(z;A),$ $ \mathcal{L}_*(z;A), F(z;A)$, and $F^q(z;A)$. 

 \begin{lemma}\label{lem:greens-function-loop-measure}
Suppose that $z \in A \subsetneq \ZZ^2$. Then
\[
G_A(z,z) =\sum_{\ell \in \mathcal{L}_*(A) : \, \ell_0=z} p(\ell)=e^{ m\left[ \mathcal{L}(z;A) \right]}=F(z;A),
\]
\[
G_A^q(z,z) =\sum_{\ell \in \mathcal{L}_*(A) : \, \ell_0=z} q(\ell)=e^{ m^q\left[ \mathcal{L}(z;A) \right]}=F^q(z;A).
\]
\end{lemma}
\begin{proof}
See \cite[Lemma~9.3.2]{lawler_limic}.  Although that book
only studies positive measures, the proof is entirely algebraic
and holds when using the weight $q$ as well.    Let us generalize the proof here.
 
Suppose $L$ is a set of nontrivial
loops rooted at a point $z\in A$ with the
property that if $\ell_1,\ell_2 \in L$, then $\ell_1 \oplus\ell_2 \in \ L$.
Let $L_1$ be the set of elementary loops, that is, the set of loops
in $L$ that cannot be written as $\ell_1 \oplus \ell_2$ with
$\ell_1,\ell_2 \in L$.  Let $L_k$ denote the set of loops
of the form  
\[       \ell = \ell_1 \oplus \cdots \oplus \ell_k, \;\;\; \ell_j \in L_1 . \]
Then $L = \bigcup_{k=1}^\infty L_k$.  Suppose now, as is true in the case we will be considering, that every loop in $L$ admits a unique (up to translation) such decomposition into concatenated elementary loops. In this case there is a unique $k$ such that $\ell \in L_k$. We may then consider the measure
$\lambda$
on loops in $L$ that assigns measure $k^{-1} q(\ell) $
to $\ell \in L_k$.  
Let $L'$ denote the set of unrooted loops that have at least
one representative in $L$.
 Then the measure $\lambda$ viewed as a measure
on unrooted loops is the same as the loop measure of unrooted
loops $[\ell]$ restricted to $L'$. 
Indeed, this measure assigns measure $ (j/k) \, q([\ell])$ to
$[\ell]$ where $j$ is the number of distinct loops  among
\[   \ell_1 \oplus \ell_2 \oplus \cdots \oplus \ell_k,\;\; \ell_2 \oplus
\ell_3 \oplus  \cdots \oplus
\ell_k \oplus \ell_1,\;\; \cdots, \;\;\ell_{k} \oplus \ell_1 \oplus \cdots
  \oplus \ell_{k-1}. \]
If $|q(L_1)| < 1$, as in our particular case, then
\begin{equation}\label{72615a}
        m^q[L']  = \sum_{k=1}^\infty \frac 1k \, q(L_k)
 = \sum_{k=1}^\infty \frac 1k \, q(L_1)^k
     = -\log \left[1 - q(L_1) \right]. 
\end{equation}
     
  In our particular case, if  $L_1$ is the set of nontrivial
   loops $\ell$ starting
 and ending at $z$, staying in $A$ and otherwise not visiting
 $z$,  then it is not hard to see that
\begin{equation}\label{72615b}
   G^q(z,z;A) = \frac{1}{1 - q(L_1)}. 
   \end{equation}
Combining \eqref{72615a} and \eqref{72615b} concludes the proof.  
\end{proof}

  \subsection{Loop-erased Random Walk}
  Let $\omega = [\omega_0, \ldots, \omega_\tau]$ be a walk with $\tau < \infty$. We say that $\omega$ is self-avoiding if $i \neq j$ implies that $\omega_i \neq \omega_j$. The {\em loop-erasure} of $\omega$, $\operatorname{LE}[\omega]$, is a self-avoiding walk defined as follows:
  \begin{itemize} \setlength{\itemsep}{5pt}
  \item{If $\omega$ is self-avoiding, set $\operatorname{LE}[\omega] = \omega$.}
  \item{Otherwise, set $s_0 = \max\{ j \le \tau : \omega_j=\omega_0 \}$ and let $\operatorname{LE}[\omega]_0 = \omega_{s_0}$;}
  \item{For $i \geq 0$, if $s_i < \tau$, set $s_{i+1} = \max\{j \le \tau: \omega_j = \omega_{s_i}\}$ and let $\operatorname{LE}[\omega]_{i+1} = \omega_{s_i+1}$.} 
  \end{itemize}
  We can now define the ``loop-erased $q$-measure" of walks $\eta$ staying in $A$:
  \begin{equation}\label{def:loop-erased-q-measure}
  \hat{q}(\eta;A) := \sum_{\omega \subset A : \, \operatorname{LE}[\omega] = \eta} q(\omega) = q(\eta) F^q(\eta; A), 
  \end{equation}
  and we define $\hat{p}$ in the same manner, replacing $q$ by $p$. (We will often omit writing out $A$ explicitly so that $\hat{q}(\eta)=\hat{q}(\eta;A)$ where no confusion is possible.) 
  Note that this quantity is zero if $\eta$ is not self-avoiding. The second identity in \eqref{def:loop-erased-q-measure} is
  proved in \cite[Proposition~9.5.1]{lawler_limic}  by observing that one can write any walk $\omega$ as a concatenation of the loops erased by the loop-erasing algorithm and the self-avoiding segments of $\operatorname{LE}[\omega]$ ``between'' the loops, and then using \eqref{lem:FVA} and Lemma~\ref{lem:greens-function-loop-measure}.  Again, although \cite{lawler_limic} deals with positive weights, the proof is equally valid when using the weight $q$.

\subsection{Fomin's identity}\label{fominsection}
What we call Fomin's identity  
is a generalization for LERW of a well-known result of Karlin-McGregor. It is a combinatorial identity that, informally speaking, allows one to express a loop-erased quantity as a determinant of random walk quantities, the latter being easier to estimate. See \cite{fomin-lerw} or Chapter 9.6 of \cite{lawler_limic} for more information and additional references. 
We state here the particular case of Fomin's identity which we will need. Let $A \in \mathcal{A}$ with two marked edges $a,b \in \p_e A$, and set $K = A \setminus [0,1]$. The idea of the proof of the lemma below is to construct a bijection between the set of (pairs of) walks that run from $b$ to $1$ and intersect the loop-erasure of a walk from $a$ to $0$ and the set of (pairs of) walks that run from $b$ to $0$ and intersect the loop-erasure of a walk from $a$ to $1$. The existence of this bijection implies that the corresponding terms cancel in the expression on the right-hand side of \eqref{eq:fomin} and the result is the expression on the left-hand side.  
\begin{lemma}
If $A \in \mathcal{A}, a,b \in \p_e A$, and $K = A \setminus [0,1]$, then for walks $\omega^1, \omega^2$ that start and end on the boundary of $K$ and otherwise stay in $K$, we have
\begin{equation}\label{eq:fomin}
\begin{split}
\sum_{\rwpath {\omega^1} {a} 0 }   & \sum_{\substack{\rwpath {\omega^2} {b} 1  \\ \omega^2 \cap \operatorname{LE}[\omega^1] = \emptyset}} p(\omega^1)p(\omega^2) - \sum_{\rwpath {\omega^1} {a} 1 }  \sum_{\substack{\rwpath {\omega^2} {b} 0  \\ \omega^2 \cap \operatorname{LE}[\omega^1] = \emptyset}}p(\omega^1)p(\omega^2)  \\
& = \sum_{\rwpath {\omega^1} {a} 0 }\sum_{\rwpath {\omega^2} {b} 1 }p(\omega^1)p(\omega^2) - \sum_{\rwpath {\omega^1} {a} 1 }\sum_{\rwpath {\omega^2} {b} 0 }p(\omega^1)p(\omega^2). 
\end{split}
\end{equation}
\end{lemma}

\subsection{Brownian loop measure}
The random walk loop measure $m$  defined in a previous section has a conformally invariant scaling limit, the Brownian loop measure $\mu$. It is a sigma-finite measure on equivalence classes of continuous loops $\omega : [0, t_\omega] \to \mathbb{C}, \, \omega(0)=\omega(t_\omega)$, with the equivalence relation given by $\omega_1 \sim \omega_2$ if there is $s$ such that $\omega_1(t) = \omega_2(t+s)$ (with addition modulo $t_{\omega_2}$); see \cite{loop-soup}. One can construct $\mu$ via the Brownian bubble measure $\mubub$: a bubble in a domain $D$, rooted at $a \in \p D$, is a continuous function $\omega: [0, t_{\omega}] \to \C$ with $\omega(0) = \omega(t_{\omega}) =a \in \partial D$ and $\omega(0, t_{\omega}) \subset D$. The bubble measure is conformally covariant (with scaling exponent $2$, see \eqref{def:bubble-covariance}) so it is enough to specify the scaling rule and give the definition in one reference domain, say $\DD$. Let 
\begin{equation}\label{poissondisk}
h_\DD(z,a)=\frac{1}{2\pi} \frac{1 - |z|^2}{|z-a|^2}
\end{equation}
be the Poisson kernel of $\DD$. Note that the Poisson kernel is conformally covariant (which is easily checked, but see Section 2.3 of \cite{Lawler_cip} for a proof).
Let $\Prob^{z,a}$ be the law of an $h$-process derived from Brownian motion and the harmonic function $h_\DD(z,a)$ (see below); informally, this $h$-process is a Brownian motion from $z$ conditioned to exit $\DD$ at $a$. We define
\begin{equation}\label{def:bubble}
\mubub_\DD(1) = \pi \lim_{\ee \to 0}  \ee^{-1} h_\DD(1-\ee,1) \Prob^{1-\ee,1}.
\end{equation}
The $\pi$ factor is present to match the notion of \cite{loop-soup} and is chosen so that the measure of bubbles in $\mathbb{H}$ rooted at $0$ that intersect the unit circle equals $1$. See also Chapter~5 of \cite{Lawler_cip} for a discussion of the suitable metric spaces on which these measures are defined.
Suppose $\varphi: \DD \to D$ is a conformal map and that $\p D$ is locally analytic at $\varphi(1)$. Then if we write
\[
\varphi \circ \mubub_\DD(1)[A] := \mubub_\DD(1)[ \{\omega : \varphi \circ \omega \in A \} ],
\]
we have the following scaling rule
\begin{equation}\label{def:bubble-covariance}
\varphi \circ \mubub_\DD(1) = |\varphi'(1)|^2\mubub_{D}(\varphi(1)).
\end{equation}
We can now define the Brownian loop measure restricted to loops in $\DD$ as the measure on unrooted loops induced by
\begin{equation}\label{def:brownian-loop-measure}
\mu = \frac{1}{\pi} \int_0^{2\pi} \int_0^1 \mubub_{r \DD}(r e^{i\theta}) r dr d\theta.
\end{equation}
The Brownian loop measure in other domains can then be defined by conformal invariance. 

The next lemma makes precise that the random walk and Brownian loop measures are close on large enough scales.

\begin{lemma}\label{lem:coupling}
There exist constants $\theta > 0$ and $c_1 < \infty$ and for all $n$ sufficiently large a coupling of the Brownian and random walk loop measures, $\mu$ and $m$, respectively, in which the following holds. There is a set $\mathcal{E}$ whose complement has measure at most $e^{-\theta n}$ and on $\mathcal{E}$ we have that all pairs of loops $(\omega, \ell)$ ($\omega$ Brownian loop and $\ell$ random walk loop) with \[ \diam \omega \ge e^{n(1-2\theta)}  \quad \text{ or } \quad \diam \ell \ge e^{n(1-2\theta)}\] satisfy \[||\omega -\ell|| \le c_1 n,\]
where $||\omega - \ell|| = \inf_\alpha ||\omega \circ \alpha - \ell||_\infty$ with the infimum taken over increasing reparameterizations.  
\end{lemma}

This result can be derived from the main theorem of  \cite{ltf}. However, let us
sketch the argument.  Another way to construct the Brownian loop measure is by
the following rooted measure.  Suppose $\omega:[0,t_\omega]\rightarrow \C$
is a loop, that is, a continuous function with $\omega(0) = t_\omega$.  We
can describe any loop $\omega$ as a triple $(z,t_\omega,\tilde \omega)$ where
$z \in \C, t_\omega > 0$ and $\tilde \omega[0,t_\omega] \rightarrow \C$ is a loop 
with $\tilde \omega(0) = \tilde \omega(t_\omega) = 0$.  The loop $\omega$ is obtained
from $(z,t_\omega,\tilde \omega)$ by translation.
We consider the measure on $(z,t_\omega,\tilde \omega)$ given by
\begin{equation}  \label{dec11.1}
    {\rm area} \times (2\pi t)^{-2} \, dt \times ({\rm bridge}_t)
    \end{equation}
where ${\rm bridge}_t$ means the probability measure associated to two-dimensional
Brownian motions  $B_t, 0 \leq s\leq t$ conditioned so that $B_0 = B_t = 0$.
The factor $(2 \pi t)^{-2}$ can be considered as $t^{-1} \, p_t(0,0)$.
where $p_t$ is the transition kernel for a two-dimensional Brownian motion.
This measure, {\em considered as a measure on unrooted loops}, is the same as the
measure
\[
\mu = \frac{1}{\pi} \int_0^{2\pi} \int_0^\infty \mubub_{r \DD}(r e^{i\theta}) r dr d\theta.
\]
The expression for $\mu$  associates to each unrooted loop the rooted loop obtained
by choosing the point farthest from the origin.  The expression \eqref{dec11.1}
chooses the root using the uniform distribution on $[0,t_\omega]$.

Similarly, a rooted random walk loop can be written as $(z,2n,l)$ where
$l$ is a loop with $\ell(0) = 0$ and $|\ell| = 2n$.  Then the measure on such triples is
\[     (\mbox{counting measure}) \times (2n)^{-1} \, \Prob\{S_{2n} = 0 \}
   \times ( \mbox{bridge}_n) .\]
   Here $S_n$ is a simple random walk starting at the origin, and $\mbox{bridge}_n$
   denotes the probability measure on $[S_0,S_1,\ldots,S_{2n}]$ conditioned
   that $  S_{2n}=0$.  Using the relation  $\Prob\{S_{2n} = 0 \}
   = (\pi n)^{-1} + O(n^{-2}) $, we can now see our coupling of the two components.
   For the first component, the root, we couple Brownian loops rooted at $\dsquare_z$
   with random walk loops rooted at $z$.   We couple Brownian loops with
   time duration  $n-\frac 12 \leq t_\ell < n  + \frac 12$ with random
   walk loops of time duration $ 2n$.  Then we use a version of the KMT
   coupling (see Theorem \ref{kmt}) of the random walk
   and Brownian loops to couple the paths.  One can then check that this coupling has the desired properties.

\subsection{KMT coupling}

We will use in a number of places the  Koml\'os, Major, and Tusn\'ady (KMT) coupling of random
 walk and Brownian motion.   For a proof of the one-dimensional case, see~\cite{kmt2} or~\cite{lawler_limic}
and  the two-dimensional case follows using a standard
trick~\cite[Theorem 7.6.1]{lawler_limic}.

\begin{theorem}\label{kmt}
There exists a coupling of planar Brownian motion $B$ and two-dimensional simple random walk $S$ with $B_0=S_0$, and a constant $c>0$ such that for every $\lambda >0$, every $n\in\R_+$, 
$$\pp \bigg(\sup_{0\leq t\leq  n \vee
T_n \vee \tau_n}|S_{2t}-B_t| > c(\lambda + 1)\log n\bigg)\leq cn^{-\lambda}.$$ 
where $T_n = \min\{t: |S_{2n} | \geq n \},
\tau_n = \min\{t: |B_t| \geq n \} $.

\end{theorem}

\section{The combinatorial identity: proof of Theorem~\ref{thm:det-formula}} \label{sect:det}
This section states and proves Theorem~\ref{thm:det-formula2} which is a more general version of Theorem~\ref{thm:det-formula}. For the statement of the theorem some more notation is needed. 

Fix a discrete domain $A\in \dsets$. Recall the definition of $J,Q$ and $q$ from Section~\ref{sect:key} and that our branch cut is $\beta = f_A^{-1}([0,1])$ which runs from $w_0$ to $\p D_A$ in $D_A$. We will assume that $r_A$ is sufficiently large so that $[0,1] \cap \beta = \emptyset$. This is possible, since $f_A'(w_0)>0$.

Let
\[
\lambda=[x_0, \ldots, x_k] \subset A\] be a self-avoiding walk (SAW) containing the ordered edge $[0,1]$  with $\dist(0,\partial D) \ge 2 \diam \lambda$. Given $\lambda$ write
\[
\lambda^R = [x_k, \ldots, x_0] 
\]
for its time-reversal. Note that $\lambda^R$ contains the ordered edge $[1,0]$. For given $a=[a_-,a_+], b=[b_-,b_+] \in \partial_e A$ we make the following definitions.
\begin{itemize} \setlength{\itemsep}{5pt}
\item{Let $\saws(\lambda)^+=\saws^+(a, b; \lambda, A)$ be the set of SAWs from $a$ to $b$ in $A$ that contain the walk $\lambda$. That is, $\saws(\lambda)^+$ consists of walks $\eta \in \saws ^+$ that can be written as $\eta^1 \oplus \lambda \oplus \eta^2$, where $\eta^1, \eta^2$ are SAWs connecting $a$ with $x_0$ and $x_k$ with $b$, respectively.}
\item{Let $\saws(\lambda)^-=\saws^+(a, b; \lambda^R, A)$ be the set of SAWs from $a$ to $b$ in $A$ that contain the reversal of
  $\lambda$.}
\item{Let $\saws(\lambda)=\saws(a,b;\lambda,A) = \saws(\lambda)^+ \cup \saws(\lambda)^-$.}
\end{itemize}
We will sometimes suppress the dependence on $\lambda$ and write just $\saws^+, \saws^-$, and $\saws$ for $\saws^+(\lambda), \saws^-(\lambda)$, and $\saws(\lambda)$; this should not cause confusion. For topological reasons (see \cite{lawler_lerw_prob} for a detailed argument), every
self-avoiding path $\eta$ from $a$ to $b$ traversing the ordered
edge $[0,1]$ yields the same value of $Q(\eta)$.  Moreover, if $\eta'$
is another SAW from $a$ to $b$ traversing $[1,0]$, then
$Q(\eta') = - Q(\eta)$. 
Indeed, consider $\zeta$ to be any boundary arc connecting $a$ to $b$. Then one of the loops $\eta \oplus \zeta$ and $\eta' \oplus \zeta$ winds around $w_0$ exactly once and the other does not, so $Q(\eta \oplus \zeta)+Q(\eta' \oplus \zeta)=0$, implying (see \eqref{5-31}) that $Q(\zeta)(Q(\eta)+Q(\eta'))=0$.
  Without loss of generality, we
  will assume that $a, b$ are labelled in such a way that
\[
\eta \in \saws(\lambda)^+ \Longrightarrow Q(\eta) = +1 ; \quad \eta \in \saws(\lambda)^- \Longrightarrow Q(\eta) = -1.
\]

Recall that $\mathcal{J}_A$ is the set of unrooted random walk loops in $A$ with odd winding number about $w_0$.

Set $K=A \setminus \lambda$ and define 
\begin{align*}
\Delta_{K}\left(x_0 \to a, x_k \to b\right) & = H_{\p K}(x_0, a)H_{\p K}(x_k, b) - H_{\p K}(x_0, b)H_{\p K}(x_k, a)\\
\Delta^q_{K}\left(x_0 \to a, x_k \to b\right) & = H_{\p K}^q(x_0, a)H_{\p K}^q(x_k, b) - H_{\p K}^q(x_0, b)H_{\p K}^q(x_k, a),
\end{align*}
where 
\[
H_{\p K}(x, a) = \sum_{\substack{\rwpath \omega x a\\
\omega \subset K}}p(\omega), \quad H_{\p K}^q(x,a) = \sum_{\substack{\rwpath \omega x a \\ \omega \subset K}}q(\omega),
\]
are the boundary Poisson kernels with the sums taken over walks started from the vertex $x$, taking the first step into $K$ and then exiting $K$ using the edge $a \in \p_e A$. Notice that $\Delta_K, \Delta^q_K$ can be written as determinants.

\begin{theorem}\label{thm:det-formula2}
Under the assumptions above,
$$\sum_{\eta \in  \saws} \hat{p}(\eta) = p(\lambda) e^{2m[\mathcal{J}_A]}F^q(\lambda; A)\left|\Delta^q_K\left(x_0 \to a, x_k \to b\right)\right|,
$$
where $\saws = \saws(\lambda)$.
In fact, 
\begin{multline*}
\sum_{\eta \in \saws^{+}}
\hat{p}(\eta) = \frac{p(\lambda)}{2} \Big[ e^{2m[\mathcal{J}_A]}F^q(\lambda; A)|\Delta^q_K\left(x_0 \to a, x_k \to b\right)| \\ + F(\lambda;A)\Delta_K\left(x_0 \to a, x_k \to b\right)\Big]
\end{multline*}
and
\begin{multline*}
\sum_{\eta \in \saws^{-}} \hat{p}(\eta) = \frac{p(\lambda)}{2} \Big[ e^{2m[\mathcal{J}_A]}F^q(\lambda; A)|\Delta^q_K\left(x_0 \to a, x_k \to b\right)|  
\\- F(\lambda;A)\Delta_K\left(x_0 \to a, x_k \to b\right)\Big],
\end{multline*}
where $\saws^+ = \saws^+(\lambda)$ and $\saws^-=\saws^-(\lambda)$.
\end{theorem}
We will use this theorem only in the
 special case when $\lambda = [0,1]$ where in the notation of the introduction the theorem gives
\[
\sum_{\eta \in  \saws }\hat{p}(\eta) = \frac{1}{4}F^q([0,1]; A)\, e^{2m[\mathcal{J}_A]}\, \left|R_A(0,a)R_A(1,b) - R_A(0,b)R_A(1,a) \right|.
\]
If we divide both sides of this equation by $H_{\p A}(a,b)$ we get \eqref{de} as stated in the introduction.

Before proving Theorem~\ref{thm:det-formula2} we need a lemma. 
\begin{lemma}\label{lem:F-Fq}
Let $\rwpath \eta a b$, $a,b \in \p_e A$, be a SAW in $A$ containing the (unordered) edge $[0,1]$. Then
\[F^q(\eta;A) = F(\eta;A) \exp\{-2m[ \mathcal{J}_A ] \},\]
where $\mathcal{J}_A$ is the set of unrooted loops in $A$ with odd winding number about $w_0$.
\end{lemma}
\begin{proof}
For a random walk loop $\ell$, let $\operatorname{wind}(\ell)$ denote its winding number about $w_0$.
Then the definition of $q$ implies that 
\begin{align*}
m^q\left( \left\{ \ell \subset A : \, \ell \cap \eta \neq \emptyset  \right\} \right)= & \, m\left(  \left\{ \ell  \subset A : \, \ell \cap \eta \neq \emptyset,\,  \operatorname{wind}(\ell) \text{ even} \right\} \right) \\
& - m\left( \left\{ \ell \subset A : \, \ell \cap \eta \neq \emptyset,\,  \operatorname{wind}(\ell) \text{ odd} \right\}\right) \\
= & \, m\left(  \left\{ \ell  \subset A : \, \ell \cap \eta \neq \emptyset \right\} \right) \\
& - 2m\left( \left\{ \ell \subset A : \, \ell \cap \eta \neq \emptyset,\,  \operatorname{wind}(\ell) \text{ odd} \right\} \right) .
\end{align*}
But any loop with odd winding number must separate $w_0$ from $\partial A$, and so intersect every SAW $\rwpath \eta a  b$ containing $[0,1]$. This implies that
\[
m^q\left( \left\{ \ell \subset A : \, \ell \cap \eta \neq \emptyset \right\} \right)= m\left( \left\{ \ell \subset A : \, \ell \cap \eta \neq \emptyset \right\} \right) 
- 2 m\left( \left\{ \ell \subset A : \, \operatorname{wind}(\ell) \text{ odd}  \right\} \right).
\]
By exponentiating both sides we get the lemma.  
\end{proof}
\begin{proof}[of Theorem~\ref{thm:det-formula2}.]
Fix $\lambda$ as in the statement for the rest of the proof. We write $\mathcal{W},\mathcal{W}^{\pm}$ for $\saws(\lambda),\saws^{\pm}(\lambda)$. The idea is to write the sums $\sum_{\eta \in \mathcal{W}^+} \hat{p}(\eta)$ and $\sum_{\eta \in \mathcal{W}^-} \hat{p}(\eta)$ in terms of both random walk and $q$-random walk quantities via the formulas (see \eqref{def:loop-erased-q-measure} and Lemma~\ref{lem:F-Fq}) \[\hat{p}(\eta)=p(\eta)F(\eta;A), \quad F(\eta;A)=e^{2m(\mathcal{J}_A)}F^q(\eta;A),\] and the facts that $p(\eta)=\pm q(\eta), \eta \in \mathcal{W}^{\pm}$. After resummation, when we add and subtract the resulting expressions, a
determinant
  identity due to Fomin   (see Section \ref{fominsection}) can be used to write the expressions in terms of random walk determinants $\Delta_{A \setminus \lambda}$ and $\Delta^q_{A \setminus \lambda}$ that do not involve loop-erased walk quantities.

Now we turn to the details. 
Write $K = A \setminus \lambda,
\Delta = \Delta_K, \Delta^q = \Delta_K^q.$
First observe that any $\eta \in \mathcal{W}^+$ can be written as
\[
\eta = (\eta^1)^R \oplus \lambda \oplus \eta^2,
\]
where $\eta^1, \eta^2$ are nonintersecting SAWs  in $K $ connecting $x_0$ with $a$ and $x_k$ with $b$, respectively. Note that any loop intersecting $\eta$ either intersects $\lambda \subset \eta$ or it does not. Consequently,  by \eqref{lem:FVA} and \eqref{def:loop-erased-q-measure}, we can write \[\hat{p}(\eta) = p(\eta)F(\eta; A) = p(\eta)F(\lambda;A) F(\eta; K).\] Using this and the above decomposition, we see that
\begin{align}\label{eq:p-hat-1}
\sum_{\eta \in \mathcal{W}^+} \hat{p}(\eta) & = \sum_{\eta \in \mathcal{W}^+}p(\eta) F(\eta; A) \nonumber\\
& =  p(\lambda)F(\lambda;A)\sum_{\substack{{\rwpath {\eta^1} {x_0} a} \\ {\rwpath {\eta^2} {x_k} b} \\ \eta^1 \cap \eta^2 = \emptyset}}p(\eta^1)p(\eta^2)F(\eta^1 \cup \eta^2; K), 
\end{align}
where the sum is over all pairs of   nonintersecting SAWs  $\rwpath {\eta^1} {x_0} a$ and $\rwpath {\eta^2} {x_k} b$ in $K $. Similarly, any $\eta \in \mathcal{W}^-$ can be decomposed
\[
\eta = (\eta^2)^R \oplus (\lambda)^R \oplus \eta^1,
\]
where $\rwpath {\eta^2} {x_k} a$ and $\rwpath {\eta^1} {x_0} b$  are nonintersecting SAWs  in $K$. We see that
\begin{equation}\label{eq:p-hat-2}
\sum_{\eta \in \mathcal{W}^-} \hat{p}(\eta) = p(\lambda)F(\lambda;A)\sum_{\substack{\rwpath {\eta^1} {x_0} b \\ \rwpath{\eta^2} {x_k} a  \\ \eta^1 \cap \eta^2 = \emptyset} }p(\eta^1)p(\eta^2)F(\eta^1 \cup \eta^2; K).
\end{equation}
(We are only summing over paths in $K$.)
Let us now consider the sum on the right-hand side of \eqref{eq:p-hat-1}. Then using \eqref{lem:FVA} we have
\begin{align*}
\sum_{\substack{\rwpath {\eta^1} {x_0} a \\ \rwpath {\eta^2} {x_k} b \\ \eta^1 \cap \eta^2 = \emptyset}}p(\eta^1)p(\eta^2)F(\eta^1 \cup \eta^2; K) & = \sum_{\substack{\rwpath {\eta^1} {x_0} a \\ \rwpath {\eta^2} {x_k} b \\ \eta^1 \cap \eta^2 = \emptyset }}p(\eta^1)F(\eta_1; K ) p(\eta^2) F(\eta_2; K \setminus \eta_1) \\
& =\sum_{\rwpath {\omega^1} {x_0} a }  \sum_{\substack{\rwpath {\omega^2} {x_k} b  \\ \omega^2 \cap \operatorname{LE}[\omega^1] = \emptyset}}p(\omega^1)p(\omega^2),
\end{align*}
where $\omega^1: x_0 \to a$ and $\omega^2: x_k \to b$ are SAWs in $K  $.
An identical argument proves the corresponding identity (interchanging $x_0$ and $x_k$) starting from the sum in the right-hand side of \eqref{eq:p-hat-2}. If we take the difference of the two expressions,
Fomin's identity implies that we may drop the non-intersection condition: 
\begin{align*}
\sum_{\rwpath {\omega^1} {x_0} a }  \sum_{\substack{\rwpath {\omega^2} {x_k} b  \\ \omega^2 \cap \operatorname{LE}[\omega^1] = \emptyset}}& p(\omega^1)p(\omega^2) - \sum_{\rwpath {\omega^1} {x_k} a }  \sum_{\substack{\rwpath {\omega^2} {x_0} b  \\ \omega^2 \cap \operatorname{LE}[\omega^1] = \emptyset}}p(\omega^1)p(\omega^2) \\
&= \sum_{\substack{\rwpath {\omega^1} {x_0} a  \\ \rwpath{\omega^2} {x_k} b }}p(\omega^1)p(\omega^2) -  \sum_{\substack{\rwpath{\omega^1} {x_k} a  \\ \rwpath{\omega^2} {x_0}  b }}p(\omega^1)p(\omega^2) \\
& = H_{\p K}(x_0, a)H_{\p K}(x_k, b) - H_{\p K}(x_k, a)H_{\p K}(x_0, b)\\
& = \Delta \left(x_0 \to a, x_k \to b\right).
\end{align*}
(Again, we are only considering paths in $K$.) 
In other words, subtracting \eqref{eq:p-hat-2} from \eqref{eq:p-hat-1} gives
\begin{equation}\label{eq:plusminus-difference}
\sum_{\eta \in \mathcal{W}^+}\hat{p}(\eta) - \sum_{\eta \in \mathcal{W}^-}\hat{p}(\eta) = p(\lambda)\, F(\lambda;A)\, \Delta  \left(x_0 \to a, x_k \to b\right)
\end{equation}
and the right-hand side involves only random walk quantities with no non-intersection conditions.
Up to now we have not used the signed weights. The idea is to express the sum $\sum_{\eta \in \mathcal{W}^+}\hat{p}(\eta) + \sum_{\eta \in \mathcal{W}^-}\hat{p}(\eta)$ as a difference involving $\hat{q}$ to which we can apply the Fomin argument.

We first claim that
\begin{equation}\label{eq:claim1}
\sum_{\eta \in \mathcal{W}^+}\hat{p}(\eta)=\Gamma_A\sum_{\eta \in \mathcal{W}^+} \hat{q}(\eta), \quad \mbox{ where }  \Gamma_A =  \exp\{2m[ \mathcal{J}_A ] \}.
\end{equation}
To see this, recall that $\hat{q}(\eta) = q(\eta) F^q(\eta; A)$.  
We already noted that $q(\eta) = p(\eta)$ for $\eta \in \mathcal{W}^+$, so Lemma~\ref{lem:F-Fq} gives \eqref{eq:claim1}. Using that $p(\eta) = -q(\eta)$ for $\eta \in \mathcal{W}^-$, a similar argument shows that
\begin{equation}\label{eq:claim2}
\sum_{\eta \in \mathcal{W}^-}\hat{p}(\eta)=-\Gamma_A\sum_{\eta \in \mathcal{W}^-} \hat{q}(\eta).
\end{equation}
Hence, adding \eqref{eq:claim1} and \eqref{eq:claim2} gives 
\[
\sum_{\eta \in \mathcal{W}^+}\hat{p}(\eta)+\sum_{\eta \in \mathcal{W}^-}\hat{p}(\eta)=\Gamma_A\left(\sum_{\eta \in \mathcal{W}^+}\hat{q}(\eta)-\sum_{\eta \in \mathcal{W}^-} \hat{q}(\eta) \right).
\]
We can now argue exactly as in the proof of \eqref{eq:plusminus-difference} replacing $p$ by $q$; it makes no difference, and in this way we get
\begin{align}\label{eq:plusminus-sum}
\sum_{\eta \in \mathcal{W}^+}\hat{p}(\eta)+\sum_{\eta \in \mathcal{W}^-}\hat{p}(\eta)& =q(\lambda)\Gamma_A\, F^q(\lambda;A)\, \Delta ^q(x_0 \to a, x_k \to b) \\ 
& = p(\lambda)\, \Gamma_A \, F^q(\lambda;A)\, \left| \Delta ^q(x_0 \to a, x_k \to b) \right|\nonumber, 
\end{align}
where the last step uses that the left-hand side of \eqref{eq:plusminus-sum} is positive and that $|q(\lambda)| = p(\lambda)$. The theorem follows by adding and subtracting \eqref{eq:plusminus-difference} and \eqref{eq:plusminus-sum}. \end{proof}
\section{Comparison of loop measures: proof of Theorem~\ref{thm:main-thm-loop-measure}}\label{sect:loops}
In this section we prove the main estimate on the random walk loop measure by comparing it with the corresponding quantity for the Brownian loop measure.

We recall some notation. Given $A \in \mathcal{A}$, $\mathcal{J}_A$ is the set of unrooted random walk loops in $A$ with odd winding number about $w_0$. Given a simply connected domain $D \ni 0$ we write $\tilde{\J}_D$ for the set of unrooted Brownian loops with odd winding number about $0$. Let $\psi_D : \DD \to D$ be the conformal map with $\psi_D(0)=0, \psi'_D(0) > 0$ and $\psi'(0)$ is the conformal radius of $D$ from $0$. Given a lattice domain $A \subset \mathcal{A}$ with corresponding $D=D_A$ we define $r_A=r_{D_{A}}$.  
For $R > 0$, set \[\ball(R) = \{z \in \C: |z| < R \}, \quad B(R) = \{z \in \ZZ^2 : |z| < R\}.\]

 \begin{theorem}\label{thm:odd-winding-number}
There  exist $0<u, c_0,c < \infty$ such that the following holds.
Let $A \in \mathcal{A}$   and let $D_A$ be the associated simply connected domain. 
Then
\[
\left|m \left(\J_A\right) - \frac{1}{8}\log r_A - c_0 \right| \le c 
\, r_A^{-u} .
\]
\end{theorem}

Our proof does not determine the value of   $c_0$.
Before giving the proof we need a few lemmas.

\begin{lemma}\label{mudiff} There exists $ c  < \infty$ such that
if $D$ is a simply connected domain containing the origin with
conformal radius $r \geq 5$, then
\[    \left|
\mu\left(\tJ_{D} \setminus \tJ_{\Disk} \right) - \frac{\log r}{8} 
\right| \leq c \,  r^{-1}. 
\]
\end{lemma}
\begin{proof} 
For $t > 1$, let $\phi(t) = \mu\left(\tJ_{t\Disk} \setminus \tJ_{\Disk} \right)$.
Note that if $1 < t < s$, then conformal invariance
of $\mu$  implies that
$ \phi(s) = \phi(t) + \phi(s/t)$,  that is, $\phi(t) = \alpha
\log t$ for some $\alpha$.   The constant $\alpha$
can be computed, see  Proposition~4.1 of \cite{lawler_lerw_prob},

\begin{equation}   \label{sept25.1}
 \mu \left(\tJ_{t\Disk} \setminus \tJ_{\Disk}\right)  =  \frac{1}{8} \log t.
 \end{equation}

Distortion
estimates imply that there is a universal $c$ such that
\[
\psi_D\left(\ball(r^{-1}-cr^{-2})\right) \subset \Disk  \subset \psi_D\left(\ball(r^{-1} +cr^{-2} )\right)
\]
Hence, by conformal invariance   and \eqref{sept25.1},
\[ \mu\left(\tJ_{D} \setminus \tJ_{\Disk} \right) 
     = \frac 18 \, \log\left[r \pm c  \right] = \frac{1}{8}\log r + O(r^{-1}). \]
 
\end{proof}

Given Lemma \ref{mudiff} and the fact that the conformal radius of $D_A$
with respect to $0$ and $w_0$ are the same up to a
multiplicative  error
of magnitude $1 + O(r_A^{-1})$, we see that to prove Theorem \ref{thm:odd-winding-number}, it suffices
to  prove that there exists $c_0$ such that 
\[   
        m(\J_A) =
         \mu\left(\tJ_{D} \setminus \tJ_{\Disk} \right) 
          + c_0 + O(r_A^{-u}) , \quad D=D_A. \]
     If we let $k$ be the largest integer such that $e^{k+1} \Disk
     \subset D$, then we can write
     \[  m(\J_A)  - 
         \mu\left(\tJ_{D} \setminus \tJ_{\Disk} \right) 
         =   \hspace{1.4 in} \]
    \[     [m(\J_A \setminus J_{B^k})
           - \mu(\tJ_D \setminus \tJ_{\ball^k}) ] +
         \sum_{j=1}^k[ m(\J_{B^{j}} 
           \setminus  \J_{B^{j-1}}) - \mu(\tJ_{\ball^{j}}
             \setminus \tJ_{\ball^{j-1}} )], \]
      where $B^j = B(e^{j}), \ball^j = \ball(e^{j}).$
      (There are no random walk loops of odd winding number which stay in $\DD$.)
       The Koebe-$1/4$ theorem implies that
  that $(\C \setminus D)  \cap \{|z| = r\}$ is nonempty; we write $r=r_D$.  Note
  that this implies that any loop in $D$ (either random walk or
  Brownian motion) with odd winding number must intersect
  $r\Disk$. 
  
The theorem then follows from the following estimate.
The phrasing of the lemma is rather technical but the basic
idea is that the measures of the set of random walk loops and Brownian loops
with odd winding number that are in $D$
and are not contained in a smaller disk $\delta r \Disk$
are almost the same.   

We use the coupling of random walk and
Brownian loops and note that the pairs of coupled loops will have these properties unless one of the following possibilities occur, each of which will be proven to have small measure. 
\begin{itemize}\setlength{\itemsep}{5pt}
\item   The Brownian loop and the random walk loop are not
very close. The loops are coupled in such a way that this happens with small measure.
\item  One of the loops (in a coupled pair) is contained in $D$ but the other is not.  If
the loops are close this would require the loop that
is inside $D$
to be close to the boundary without intersecting it. The measure of such loops can be controlled using Beurling-type estimates.
\item  Similarly, one loop can intersect $r \Disk$ or
be contained in $\delta r \Disk$  while the other is not.  Again,
this requires one of the loops to be near a circle without
intersecting the circle. 
\item  The final ``bad'' possibility is that the loops are close but
they are so close to the origin that the winding numbers can
  differ. The measure of walks that are close to the origin is sufficiently
large that we cannot just ignore this term.  However, if a loop
(random walk or Brownian motion) gets close to the origin it
is almost equally likely to have an odd number
 as an even winding number.  This allows us to show that
 the random walk and Brownian loop measures of such loops
 are nearly the same.
 
\end{itemize}

    \begin{lemma}\label{lemma:loop-coupling}  There exist $ u > 0, c < \infty$ such
    that the following holds for all $\delta \geq 1/10$. 
    Suppose $D$ is a simply connected domain
containing the origin and let $r = r_D$.
  Let $\mu$ denote
the Brownian loop measure and $m$ the random walk loop measure.
\begin{itemize}\setlength{\itemsep}{5pt}
\item  Let $I(\ell)$ (resp., $\tilde I(\omega)$)
 be the indicator function of the event that
a random walk loop $\ell$ (resp., a Brownian loop $\omega$)
 is a subset of $D$, intersects $r \Disk$,
  is not a subset of $\delta r \Disk$.
\item  Let $
U(\ell)$  (resp., $\tilde U(\omega)$) denote the indicator function
that the winding number of $\ell$ (resp., $\omega$)
about $w_0$ (resp., $0$) is odd.  
\end{itemize}
Then,
\[     \left|\mu[\tilde I (\omega) \, \tilde U(\omega)]-
    m[I(\ell) \, U(\ell) ] \right| \leq c \, r^{-u} . \]
    \end{lemma}
    
    Here we are writing $\mu[\cdot]$ for the integral with
    respect to $\mu$ and similarly for $m[\cdot]$ and
    $\nu[\cdot]$ in the proof below.

   \begin{proof}
   We will be doing detailed
    estimates for the random walk loop measure; the Brownian loop measure estimates are done similarly. We will, however, prove one estimate for the Brownian loop measure in order to explicitly illustrate this point. The proof of this lemma will complete the proof of Theorem~\ref{thm:odd-winding-number}.
    
  It
  will be useful to fix an enumeration  $\{z_0,z_1,\ldots\}$
  of $\Z^2$ such that $|z_j|$ is nondecreasing and we let
  $V_j = \{z_0,\ldots,z_j\}$.   In
  particular, $z_0 = 0$. 
   We will first consider  loops that do not
  lie in $r^{1+u} \Disk$ for some $u > 0$.   As already noted,
  any loop in $D$ with odd winding number must intersect
  $r \Disk$. 
  
 \begin{itemize}\setlength{\itemsep}{5pt}
 \item {\bf Claim 1:} The random walk and Brownian loop
 measures of loops that intersect both $r\Disk$ and the
 circle of radius $r^{1+u}$ is $O(r^{-u})$.
 \end{itemize}
 
    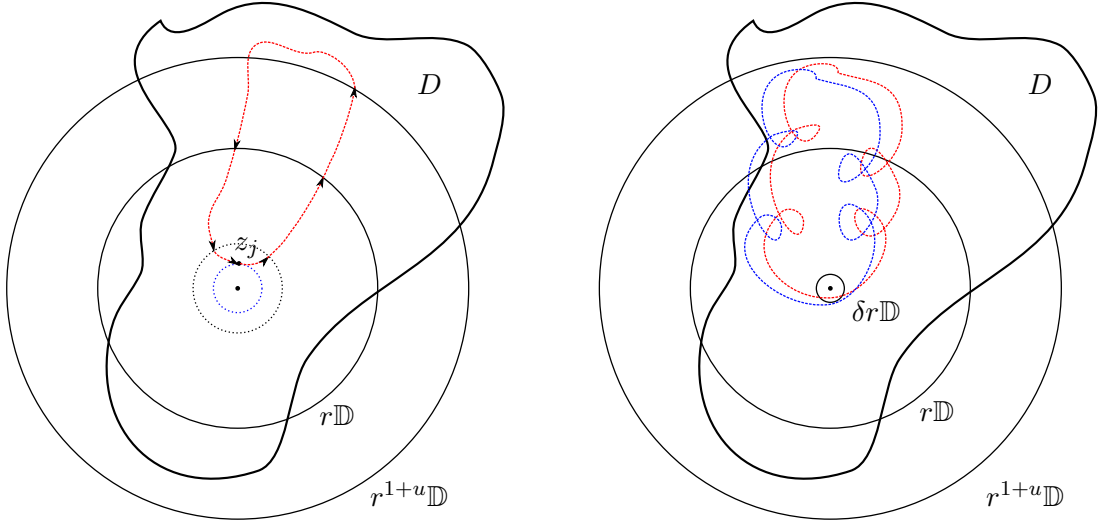
\begin{figure}[t]
  \centering
  \def\svgwidth{\columnwidth}
  
  \begingroup%
  \makeatletter%
  \providecommand\color[2][]{%
    \errmessage{(Inkscape) Color is used for the text in Inkscape, but the package 'color.sty' is not loaded}%
    \renewcommand\color[2][]{}%
  }%
  \providecommand\transparent[1]{%
    \errmessage{(Inkscape) Transparency is used (non-zero) for the text in Inkscape, but the package 'transparent.sty' is not loaded}%
    \renewcommand\transparent[1]{}%
  }%
  \providecommand\rotatebox[2]{#2}%
  \ifx\svgwidth\undefined%
    \setlength{\unitlength}{458.99781466bp}%
    \ifx\svgscale\undefined%
      \relax%
    \else%
      \setlength{\unitlength}{\unitlength * \real{\svgscale}}%
    \fi%
  \else%
    \setlength{\unitlength}{\svgwidth}%
  \fi%
  \global\let\svgwidth\undefined%
  \global\let\svgscale\undefined%
  \makeatother%
  \begin{picture}(1,0.41737352)%
    \put(0,0){\includegraphics[width=\unitlength]{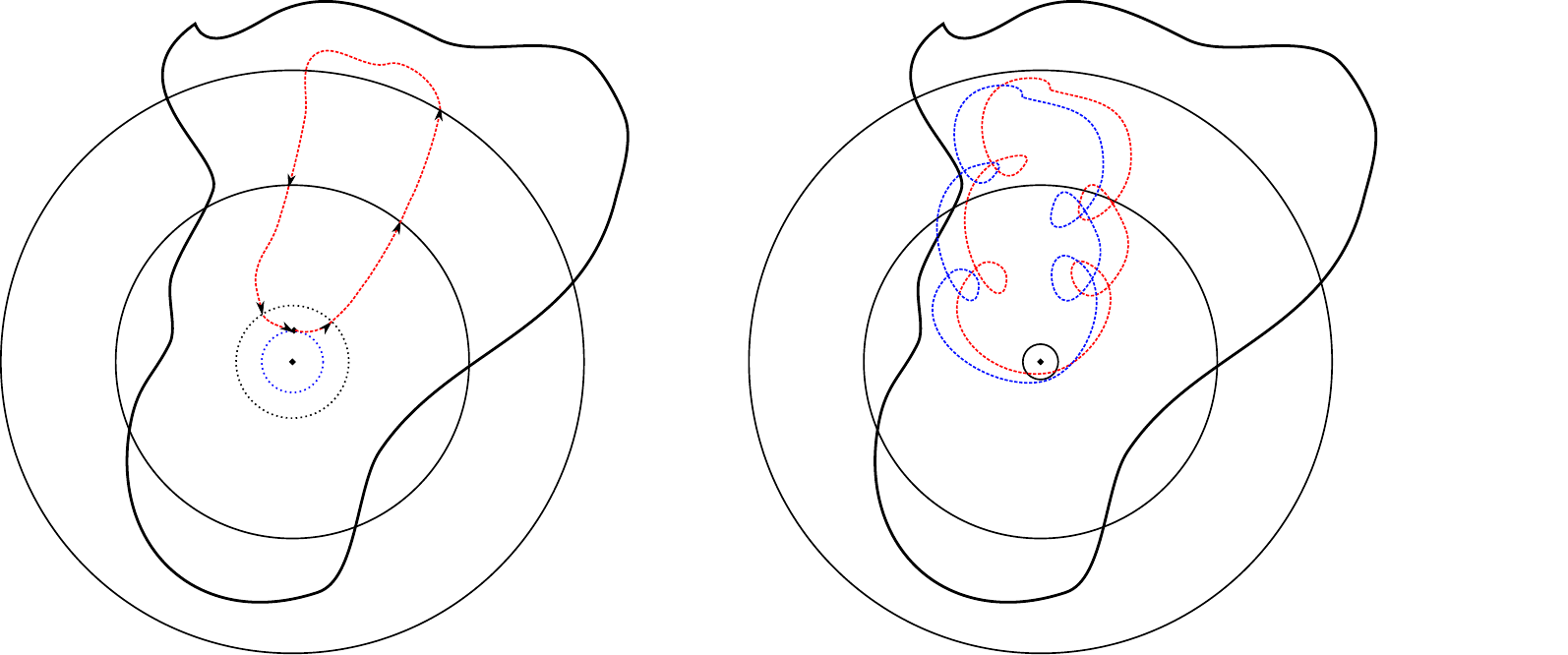}}%
    \put(0.25246737,0.076308){\color[rgb]{0,0,0}\makebox(0,0)[lb]{\smash{$r \mathbb{D}$}}}%
    \put(0.68092801,0.16017951){\color[rgb]{0,0,0}\makebox(0,0)[lb]{\smash{$\delta r \mathbb{D}$}}}%
    \put(0.7354203,0.07946977){\color[rgb]{0,0,0}\makebox(0,0)[lb]{\smash{$r \mathbb{D}$}}}%
    \put(0.29315267,0.01151761){\color[rgb]{0,0,0}\makebox(0,0)[lb]{\smash{$r^{1+u}\mathbb{D}$}}}%
    \put(0.7893083,0.01151761){\color[rgb]{0,0,0}\makebox(0,0)[lb]{\smash{$r^{1+u}\mathbb{D}$}}}%
    \put(0.33109416,0.34350539){\color[rgb]{0,0,0}\makebox(0,0)[lb]{\smash{$D$}}}%
    \put(0.82275221,0.34350539){\color[rgb]{0,0,0}\makebox(0,0)[lb]{\smash{$D$}}}%
    \put(0.18376353,0.2173314){\color[rgb]{0,0,0}\makebox(0,0)[lb]{\smash{$z_j$}}}%
  \end{picture}%
\endgroup%

  \caption{Proof of Lemma~\ref{lemma:loop-coupling}. Left: Proof of Claim 1. If $\ell$ is a good loop starting from $z_j$, then it must first exit the disk of radius $2|z_j|$ without hitting $V_{j-1}$ (represented by the smallest disk), get to $\p r \mathbb{D}$, exit $r^{1+u} \mathbb{D}$, then get back to $r \mathbb{D}$ all while not exiting $D$, and finally return to $z_j$ avoiding $V_{j-1}$. Right: Proof of Claim 2. A random walk and Brownian loop are coupled and close, inside $r^{1+u} \mathbb{D}$ but $|I-\tilde{I}| > 0$. One of three possibilities why the latter may hold is that exactly one of the loops exits $D$. In that case the other loop must get near $\p D$ without exiting $D$. }\label{loops}
\end{figure}
 
 We will prove this for the random walk measure; the Brownian
 loop estimate is done similarly. 
Let $L$ denote the set of random walk
  loops in $D$ that intersect both $r\Disk$ and the circle
  of radius $r^{1+u}$.  Then
  \[   L = \bigcup_{|z_j| < r} L_j^* , \]
where  $L_j^*$ denotes the set of such loops $[\ell]$ such
  that $z_j \in \ell $ and $\ell \cap  V_{j-1}
   = \emptyset$.  For any
  $[\ell]  \in L_j^*$ we call a rooted representative 
  \[  \ell = [\ell_0,\ell_1,\ldots,\ell_{2n} ] \]
 good if $\ell_0 = z_j$ and if we define $k$ to be the first index with $|\ell_k|
  \geq 
 r^{1+u}$, then $\ell_s \neq z_j, 1 \leq s \leq k$.  Let $L_j^i$ denote
 the set of unrooted loops that have $i$ good representatives.  Then
 $L_j^1$ is a set of elementary loops as in the proof
   of Lemma \ref{lem:greens-function-loop-measure}.  In particular,
   \[   m(L^*_j) = - \log\left[1 - p(L_j^1) \right] \]

 For $j = 0$, a good loop must start from $0$, then reach the disk of radius $r$
 without returning to $0$.  This has probability $O(1/\log r)$.  Given
 this, the Beurling estimate implies that the probability of reaching the
 disk of radius $r^{1+u}$ without leaving $D$ is $O(r^{-u/2})$.  Given
 this, the probability to return to the disk of radius $r$ without leaving $D$
 is $O(r^{-u/2})$.  Given this, the expected number of returns to the origin
 before leaving $D$ is $O(1)$.  Combining all of these estimates, we see that
 $p(L_0^1) = O(r^{-u}/\log r)$ and hence $m(L^*_0) = O(r^{-u}/\log r).$

 Let us now consider elementary loops for $j > 0$. Let $x = |z_j|$.
   Using the gambler's
 ruin estimate, the probability that the random walk reaches
 the circle of radius $2x$ without
 hitting $V_{j-1}$ is $O(1/x)$.  Given this, the probability that the
 walk reaches the circle of radius $r$ without hitting $V_{j-1}$
  is $O(1 \wedge[\log (r/x)]^{-1})$.
 Given this, as above, the probability to reach the circle of radius $r^{1+u}$
 and return to the circle of radius $r$ without leaving $D $ is $O(r^{-u})$.
 Given this, the probability
to reach within distance $2x$ of $z_j$ is $O(1 \wedge[\log (r/x)]^{-1})$.
Given this, the probability that the next visit to $V_j$ is at $z_j$ is 
$O(x^{-1})$.
  Given this, the expected number of visits to $z_j$ before leaving
$D\setminus
 V_{j-1}$ is $O(1)$.  Combining all of these estimates, we see that 
\begin{equation}  \label{jan9.3}
 m(L_j^*)  = p(L_j^1) \, \left[1 +  O(p(L_j^1) ) \right]
  \leq   \frac{c}{|z_j|^2\,r^u} \,   \left[1 \wedge \frac{1}{\log^2(r/|z_j|)} 
   \right].
  \end{equation}
By summing over $|z_j| \leq r$, we get $m(L) = O(r^{-u})$ which
establishes Claim 1. 
 
\medskip

     We will now use the coupling as described
   in Lemma \ref{lem:coupling}.  Let us
   write $(\omega,\ell)$ for a Brownian motion/random walk
   loop pair. 
    This coupling
   defines $(\omega,\ell)$ on the same measure space
   $(M,\nu)$ such that 
   \begin{itemize}\setlength{\itemsep}{5pt}
   \item  The marginal measure on $\omega$ restricted to
   nontrivial loops is $\mu$.
   \item  The marginal measure on $\ell$ restricted to
   nontrivial loops is $m$.
   \item   Let $E$ denote the set of  $(\omega,\ell)$ such that
   at least one of the paths has diameter greater than $r^{1-u}$
   and is contained in the disk of radius $r^{1+u}$ and such
   that $\|\omega - \ell \|_\infty \geq c_0 \, \log r$.  (Here by
   $\|\cdot \|_\infty$ we mean the infimum of the supremum
   norm over all parametrizations.)
   Then $\nu(E) \leq O(r^{-u})$. 
   \end{itemize}
Given Claim 1, we see that
is suffices to 
    show that
$  |\nu   ( \tilde I 
   \, \tilde U - I \, U)|= O(r^{-u}).$
   We will write $E$ for $1_E$.
   Let $K (\ell) $ (resp., $\tilde K(\omega)$)
    denote the indicator function that
   $\dist(0, \ell)  \leq  r^{1/2}$  (resp., $\dist(0,\omega)
  \leq  r^{1/2}$).    
  Note that 
  \[UI - \tilde U \tilde I = U I K - \tilde U \tilde I \tilde K
   + \tilde U \tilde I\, (\tilde K - K) + (UI -\tilde U\, \tilde I) \, (1 - K).\]
   Note that if $K =   0$
  and $\|\omega - \ell \|_\infty \leq c_0 \, \log r$, then
  (for $r$ sufficiently large) $U = \tilde U $.  
Therefore,
\[ |\nu   (   \tilde I 
   \, \tilde U - I \, U    ) |  \leq 
 |\nu(\tilde I \tilde U \tilde K ) -
     \nu( I   UK) |  +
       \nu [ \tilde I \, \tilde U    \, |\tilde K - K| ]
        + \nu [ |\tilde I  - I | \, (\tilde U + U) ]  + \nu(E)     .\]
 Therefore it suffices to establish
 \begin{itemize}\setlength{\itemsep}{5pt}
 \item {\bf Claim 2:}
 \[  \nu [ (\tilde I + I)  \, |\tilde K - K| ]
   + \nu [ |\tilde I - I | ]  \leq O(r^{-u}), \]
     \end{itemize}
     and
  \begin{itemize}\setlength{\itemsep}{5pt}    
   \item {\bf Claim 3:}
 \[  |\nu(\tilde I \tilde U \tilde K ) -
     \nu( I   UK) |  \leq O(r^{-u}).
     \]
     \end{itemize}

 To prove Claim 2, we first note that if the loops
 are coupled, and  $IU(\ell) \neq 0$,
 then $\diam(\ell) \geq \delta r$ and similarly for the Brownian
 motion loops.
Also, if  $(\omega, \ell)$ are in the disk of radius $r^{1+u}$
 with 
 $E(\omega,\ell) =0$ and  $I \neq \tilde I$, then either
 the random walk loop or the Brownian loop does one
 of the following:
 \begin{itemize}\setlength{\itemsep}{5pt}
 \item     gets within
 distance $c_0 \log r$ of $\p D$ without leaving $D$
 \item  gets within distance $c_0 \log r $ of the the circle
 of radius $r$ without hitting the circle
 \item gets within distance $c_0 \log r$ of the circle of
 radius $\delta r$ without hitting the circle.
 \end{itemize}
   We can
 estimate the measure of loops that satisfy this as well as
 $\diam \geq \delta r$,  by using the rooted loop measure. 
 We will do the random walk case for the
 first bullet; the Brownian motion case and the
 other two bullets are done similarly.
   Let $\epsilon $ be any positive number. 
Note that the root must be in the disk of radius $r^{1+u}$
and hence there are $O(r^{2(1+u)})$  choices for the root. 
Using standard large deviations estimates, except for  a set
of loops of measure $o(r^{-5})$, these loops must have
time duration at least $r^{2-\epsilon}$.   We consider
the probability that a random walk returns to the origin
at time $2n$ after getting within distance $O(\log r)$ of $\p D$
but not leaving $D$. We claim that this is $O(\log r/n^{5/4})$.
To see this, first note that  by considering
 the reversal of the walk, we can
see this is bounded above by the probability that the random
walk gets within distance $O(\log r)$ in the first $n$ steps,
stays in $\p D$, and then returns to the original point at time
$2n$.  Given that the walk is within distance $O(\log r)$ of $\p D$, the
Beurling estimate implies that the probability of not leaving $D$
in the next $n/2$ steps is $O(\log r/n^{1/4})$.   Given this,
the probability of being at the origin at time $2n$ is $O(n^{-1})$. 
The rooted loop measure puts an extra factor of $(2n)^{-1}$ in.
Therefore the rooted loop measure of loops rooted at $z$
of time duration $2n \geq r^{2-\epsilon}$ 
that have diameter at least $\delta r$, and  get within
distance $c_0 \log r$ of $\p D$ but stay in $D$ is
$O(\log r/n^{9/4})$. 
By summing over $2n \geq r^{2-\epsilon} $, we see that the rooted loop measure of loops
rooted at $z$ that have diameter at least $\delta r$, get within
distance $c_0 \log r$ of $\p D$ but stay in $D$ is $O(r^{-(2-
 2 \epsilon)5/4})$.  If we sum over $|z| \leq r^{1+u}$, we get
 that the loop measure (rooted or unrooted) of such loops is
 $O(r^{2u + \frac{5\epsilon}2  - \frac 12 }) \leq O(r^{-1/4})$ 
 for $u$ sufficiently small.   This
 gives the upper bound on  $\nu [ |\tilde I - I | ] $.
    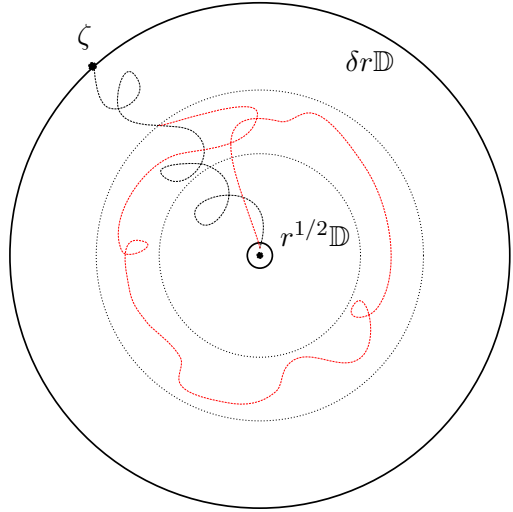
\begin{figure}[t]
  \centering

  \begingroup%
  \makeatletter%
  \providecommand\color[2][]{%
    \errmessage{(Inkscape) Color is used for the text in Inkscape, but the package 'color.sty' is not loaded}%
    \renewcommand\color[2][]{}%
  }%
  \providecommand\transparent[1]{%
    \errmessage{(Inkscape) Transparency is used (non-zero) for the text in Inkscape, but the package 'transparent.sty' is not loaded}%
    \renewcommand\transparent[1]{}%
  }%
  \providecommand\rotatebox[2]{#2}%
  \ifx\svgwidth\undefined%
    \setlength{\unitlength}{217.77807011bp}%
    \ifx\svgscale\undefined%
      \relax%
    \else%
      \setlength{\unitlength}{\unitlength * \real{\svgscale}}%
    \fi%
  \else%
    \setlength{\unitlength}{\svgwidth}%
  \fi%
  \global\let\svgwidth\undefined%
  \global\let\svgscale\undefined%
  \makeatother%
  \begin{picture}(1,0.87815748)%
    \put(0,0){\includegraphics[width=\unitlength]{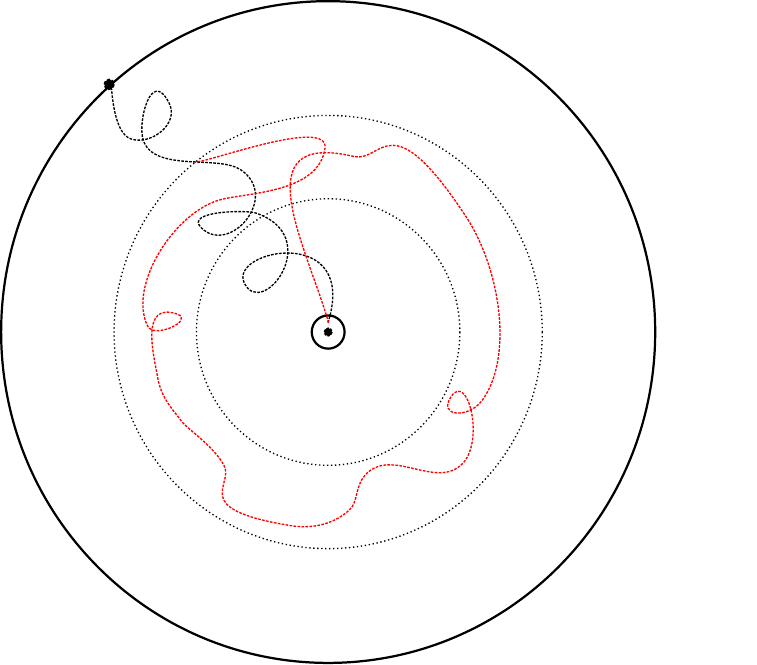}}%
    \put(0.58212548,0.75790649){\color[rgb]{0,0,0}\makebox(0,0)[lb]{\smash{$\delta r \mathbb{D}$}}}%
    \put(0.46911387,0.45591112){\color[rgb]{0,0,0}\makebox(0,0)[lb]{\smash{$r^{1/2} \mathbb{D}$}}}%
    \put(0.11812857,0.80460199){\color[rgb]{0,0,0}\makebox(0,0)[lb]{\smash{$\zeta$}}}%
  \end{picture}%
\endgroup%


  \caption{Proof of Claim 3 of Lemma~\ref{lemma:loop-coupling}. We consider loops that get close to $0$, that is, enter $r^{1/2} \mathbb{D}$ and that are not contained in $\delta r \mathbb{D}$. Starting from inside $r^{1/2} \mathbb{D}$, the first part of such a loop (which is the only part shown in the diagram) can be written as an $h$-process aiming at $\zeta \in  \delta r \p \mathbb{D}$. In each dyadic annulus the probability to change parity of the winding number is uniformly bounded from below. Hence the exit distributions (at the annulus boundary component of larger radius) of conditioned random walks/Brownian $h$-processes of odd and even winding number are uniformly comparable. This implies that the paths of odd/even winding number can be coupled with positive probability in each annulus.}\label{coupling}
\end{figure}

Similarly, if $E(\omega,\ell) =0$  , $K(\ell) \neq \tilde K(\omega)$,
 and $I(\ell) +\tilde I(\omega) \geq 1$,  then 
 \[          r^{1/2} - c_0 \, \log r \leq \dist(0,\ell)
    \leq r^{1/2} + c_0 \, \log r.\]
  and for $r$ sufficiently large,
  \[\diam(\ell) \geq \frac{\delta r}{3}.\]
 Therefore $\nu[(I +\tilde I)
 |K-\tilde K| \, (1-E)]$ is bounded above by twice
 the $m$ measure of the set of loops $\ell$
 that satisfy these conditions.  This can be estimated
 as in the previous paragraph (or using an unrooted
 loop measure estimate as in the beginning of this proof); we omit the details since we have already written out analogous estimates.
 This finishes the proof of Claim 2.

\medskip

For the final claim,
first note that 
    \[   |\nu(\tilde I  \tilde K ) -
     \nu( I   K) |
 \leq     \nu [ (\tilde I + I)  \, |\tilde K - K| ]  + \nu [ |\tilde I - I | ]  , \]
   and hence by Claim 2,
       \[   |\nu(\tilde I  \tilde K ) -
     \nu( I   K) |= O(r^{-u}).\]
 Therefore to prove Claim 3 it suffices to show that
 $$|\nu(\tilde I \tilde U \tilde K ) -
     \nu( I   UK) | = \frac 12 \, |\nu(\tilde I  \tilde K ) -
     \nu( I   K) | + O(r^{-u}).
$$
     We will prove that the Brownian loop measure of loops of odd and even winding numbers, respectively, that intersect $r^{1/2} \mathbb{D}$ and $\delta r \TT$ (we write $\TT = \p \DD$), are the same up to a small error. (This estimate for Brownian loops can be done by explicit computation, but we give an argument that also works for random walk.) Recalling \eqref{def:brownian-loop-measure} and \eqref{def:bubble} we see that it will be enough to prove this for the Brownian bubble measure of bubbles in $\Delta_s=\{|z| > s\}$ that are attached at $0 < s \le r^{1/2}$ and intersect $  \delta r\TT$; it will be enough to do the argument for $s=r^{1/2}$. Choose $\zeta \in \delta r \TT$ arbitrarily. Consider a Brownian bubble in $\Delta_{r^{1/2}}$ attached at $r^{1/2}$ that intersects $\delta r \TT$ for the first time at $\zeta$. The initial part of the bubble, the part which connects $r^{1/2}$ with $\zeta$, has the distribution of a Brownian excursion between these points in the annulus $\delta r\DD \setminus r^{1/2} \DD$. We will show that this path is about as likely to have odd as even winding number. For $k=1,2,\ldots ,\lfloor \log r^{1/2} -2 \rfloor$, let $A_k$ be the annulus with boundary components $\delta r 2^{-(k+1)} \TT$ and $\delta r 2^{-k} \TT$. Note that the probability that an excursion as above separates the two boundary components of $A_k$ between its first hitting times of the boundary components is  uniformly bounded away from $0$ independently of $k,r$. (This follows from a harmonic measure estimate for Brownian motion and for the excursion by comparing Poisson kernels.) This means the excursion has positive probability (independent of $r,k$) to change winding number parity when crossing each annulus $A_k$ and that the hitting distributions of the outer boundary of $A_k$ for excursions of odd and even winding number (up to that hitting time) are uniformly comparable. We may therefore couple two excursions from $r^{1/2}$ in such a way that with probability $1-O(r^{-u})$ they have different winding number parity when arriving at $\zeta$, e.g., as follows: first couple the two sequences of annulus hitting points (in decreasing $k$ order). This can be done so that the sequences eventually agree with large probability. (See \cite[Section 1.5]{lindvall}.) Then given the hitting points, sample the subpaths connecting these hitting points. The paths couple (and run together) after the point at which the hitting points agree and the winding number parities (at the hitting time) are different. Since constants are independent of $r,k$, the excursions couple with probability $c>0$ (independent of $r,k$) in each of the annuli, and so the paths couple with probability $1-O((1-c)^{\log r^{1/2}}) = 1-O(r^{-u})$ for $u=u(c) > 0$. This shows that the probability of the loop attached at $r^{1/2}$ having odd winding number equals the probability of the loop having even winding number up to an error of $O(r^{-u})$. The analogous argument works for the random walk loops and so we have established Claim~3, which completes the proof of the lemma as well as Theorem~\ref{thm:odd-winding-number}.          
\end{proof}

\section{Asymptotics of $\Lambda$: proof of Theorem~\ref{thm:theorem4-new}}\label{sect:spinor convergence}
In this section we study  the asymptotics of  
$
\Lambda_A(z,a) =  {R_A(z,a)}/{H_A(z,a)},  
$
as $r_A \to \infty$. We recall that $R_A(z,a) = \mathbf{E}^z[Q(S[0,\tau]) I_a]$ where $S=S^z$ is simple random walk from $z$, $I_a$ is as defined in \eqref{Ia},
and that $H_A(z,a) = \mathbf{P}^z\left(S_{\tau} = a \right)$ is discrete harmonic measure. Consequently we have
\begin{equation}\label{def:f}
\Lambda_A(z,a) =\frac{R_A(z,a)}{H_A(z,a)} = \mathbf{E}^{z,a}[Q(S[0,\tau]) I], \quad z \in A,
\end{equation} 
 where $I=1\{S[1,\tau] \cap \{0,1\} = \emptyset\}$ 
 and  $\mathbf{E}^{z,a}$ denotes
  expectation   with respect to the measure under which $S$ is a simple random walk from $z$ conditioned to exit $A$ at $a$.

\subsection{Continuum functions}\label{sect:candidate} \label{sect:spin-def}
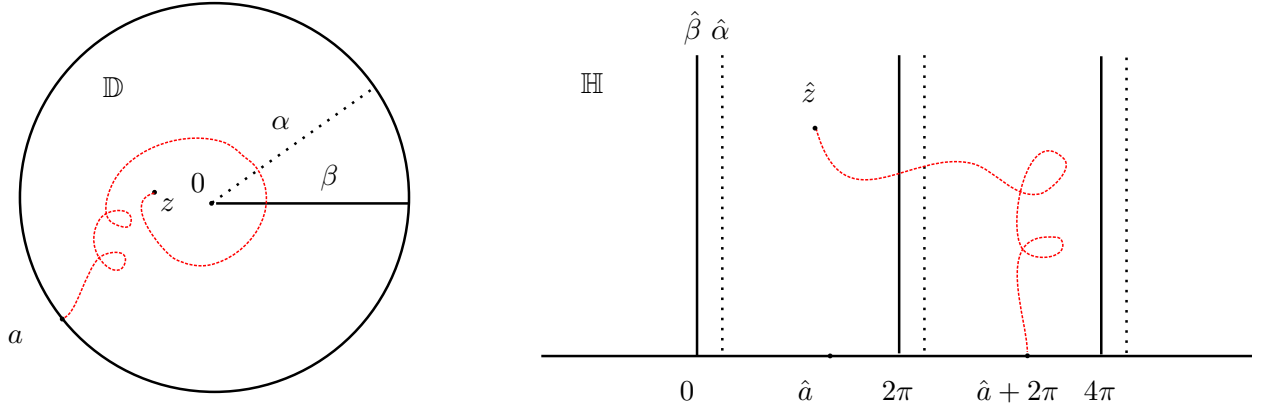
\begin{figure}[t]
  \centering
    \def\svgwidth{\columnwidth}

  \begingroup%
  \makeatletter%
  \providecommand\color[2][]{%
    \errmessage{(Inkscape) Color is used for the text in Inkscape, but the package 'color.sty' is not loaded}%
    \renewcommand\color[2][]{}%
  }%
  \providecommand\transparent[1]{%
    \errmessage{(Inkscape) Transparency is used (non-zero) for the text in Inkscape, but the package 'transparent.sty' is not loaded}%
    \renewcommand\transparent[1]{}%
  }%
  \providecommand\rotatebox[2]{#2}%
  \ifx\svgwidth\undefined%
    \setlength{\unitlength}{393.25233154bp}%
    \ifx\svgscale\undefined%
      \relax%
    \else%
      \setlength{\unitlength}{\unitlength * \real{\svgscale}}%
    \fi%
  \else%
    \setlength{\unitlength}{\svgwidth}%
  \fi%
  \global\let\svgwidth\undefined%
  \global\let\svgscale\undefined%
  \makeatother%
  \begin{picture}(1,0.474987)%
    \put(0,0){\includegraphics[width=\unitlength]{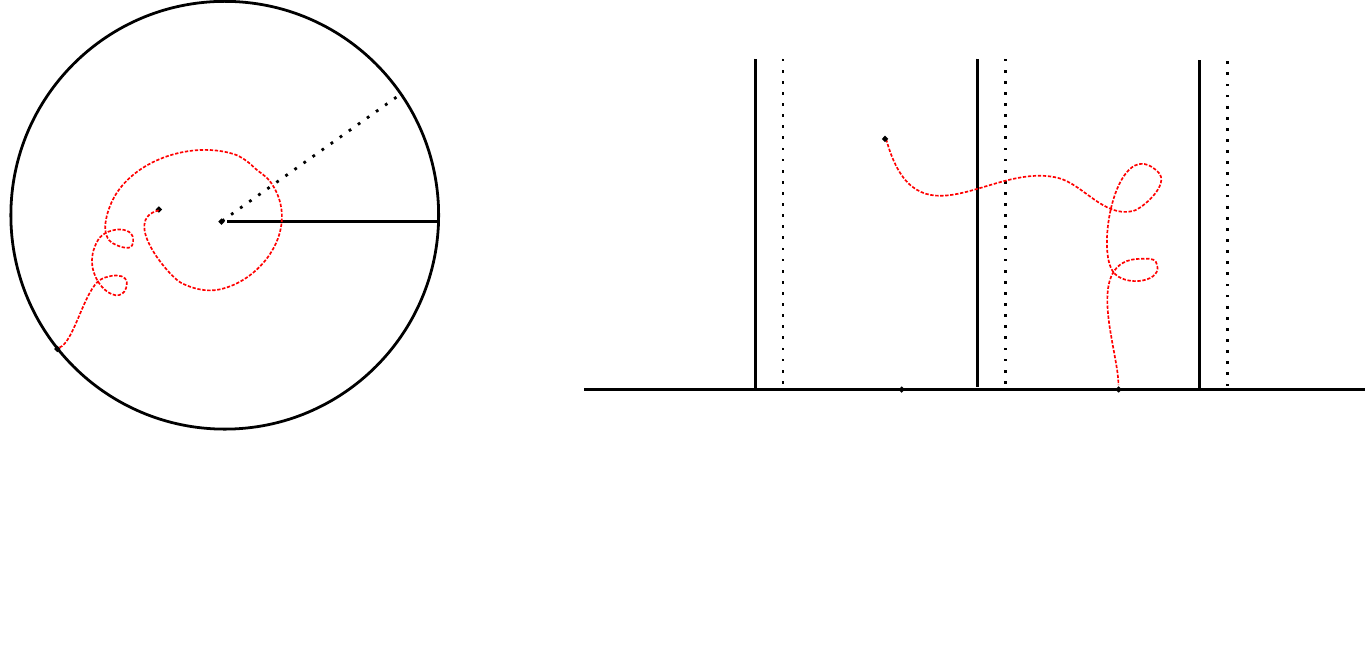}}%
    \put(0.14568935,0.32209685){\color[rgb]{0,0,0}\makebox(0,0)[lb]{\smash{$0$}}}%
    \put(0.250052,0.32764761){\color[rgb]{0,0,0}\makebox(0,0)[lb]{\smash{$\beta$}}}%
    \put(0.53922944,0.15393725){\color[rgb]{0,0,0}\makebox(0,0)[lb]{\smash{$0$}}}%
    \put(0.70197482,0.15393725){\color[rgb]{0,0,0}\makebox(0,0)[lb]{\smash{$2\pi$}}}%
    \put(0.86472014,0.15393725){\color[rgb]{0,0,0}\makebox(0,0)[lb]{\smash{$4\pi$}}}%
    \put(0.45955756,0.40135629){\color[rgb]{0,0,0}\makebox(0,0)[lb]{\smash{$\mathbb{H}$}}}%
    \put(-0.00168864,0.20038168){\color[rgb]{0,0,0}\makebox(0,0)[lb]{\smash{$a$}}}%
    \put(0.21059033,0.37266525){\color[rgb]{0,0,0}\makebox(0,0)[lb]{\smash{$\alpha$}}}%
    \put(0.07473248,0.39824319){\color[rgb]{0,0,0}\makebox(0,0)[lb]{\smash{$\mathbb{D}$}}}%
    \put(0.63734496,0.39470709){\color[rgb]{0,0,0}\makebox(0,0)[lb]{\smash{$\hat{z}$}}}%
    \put(0.12096359,0.30650663){\color[rgb]{0,0,0}\makebox(0,0)[lb]{\smash{$z$}}}%
    \put(0.63423966,0.15392421){\color[rgb]{0,0,0}\makebox(0,0)[lb]{\smash{$\hat{a}$}}}%
    \put(0.77817995,0.15456856){\color[rgb]{0,0,0}\makebox(0,0)[lb]{\smash{$\hat{a}+2\pi$}}}%
    \put(0.54249554,0.44538132){\color[rgb]{0,0,0}\makebox(0,0)[lb]{\smash{$\hat{\beta} \, \, \hat{\alpha}$}}}%
  \end{picture}%
\endgroup%

  
  \caption{The construction of the continuum observable $\lambda(z,a)$ in $\mathbb{D}$. Left: We consider a Brownian $h$-process $W$ in $\mathbb{D}$ from $z$ conditioned to exit $\mathbb{D}$ at $a$. Right: $W$ is lifted to a continuous process $\hat{W}$ in $\mathbb{H}$ by the multivalued function $F(w)=-i \log w$ starting from $\hat{z}$, the point in $F(z)$ with real part in $[0,2\pi)$. The random variable $Q$ equals $+1$ if $\hat{W}$ exits $\mathbb{H}$ in $\{\hat{a} + 4 k \pi, \, k \in \mathbb{Z}\}$ and $-1$ otherwise and $\lambda$ is the expected value of $Q$ with respect to the law of $W$. By symmetry $\lambda(z,a)=0$ for $z \in \alpha$, the antipodal line relative to $a$. For paths avoiding $\alpha$, the value of $Q$ depends only on whether $\alpha \cup \beta$ separates $z$ from $a$ in $\mathbb{D}$.}\label{lambda}
\end{figure}

Given \eqref{def:f} and the fact that simple random walk converges to Brownian motion, we would expect any scaling limit of $\Lambda_A(0,a)$ to be the corresponding quantity with random walk replaced by an $h$-transformed Brownian motion conditioned to exit $D$ at $a \in \p D$.   Here we describe the
quantity in the continuum.

We will do the construction in the unit disk $\DD$;
we can use
conformal invariance to define the functions in other simply connected
domains. Let $\beta$ denote the line segment $[0,1) \subset \Disk$.
 Let  $a = e^{2i\theta_a}, 0 < \theta_a < \pi$, and let
  $\alpha=\alpha(a)=\{w: w=re^{i(2\theta_a + \pi)}, \, r \in (0,1) \}$ be the antipodal radius relative to $a$.  
Let $\mathbf{P}^{z,a}$ be a probability measure under which the process $W_t, \, t \in [0, T],$ is a Brownian $h$-process
 in $\DD$ started from $z$ conditioned to exit $\DD$ at $a$.  Here $T$ is the hitting time of $\p \DD$.   For each realization of the process $W$,
 we, roughly speaking, let $Q$ equal $\pm 1$ depending on whether the path $W[0,T]$
 intersects $\beta$ an even or an odd number of times.  This
 is a bit imprecise since there are an infinite number of intersections.
One way to make it precise by lifting by the multi-valued logarithm
$F(z) = -i\log z$.    The image of $\beta$, $F(\beta)$, is
the union of the  $2\pi$-translates
of the positive imaginary axis.  If we choose a particular image
of $z$, say $\hat z$,
then there is a corresponding image $\hat a$ of $a$ such that
$\hat a$ is on the boundary of the connected component of
$\HH \setminus F(\beta)$ that contains $\hat z$.
Once $\hat z$ is given, the $h$-process is mapped to an
$h$-process $\hat W_t$ in $\Half$ conditioned to leave $\Half$
at $F(a) = \{\hat a + 2\pi k: k \in \Z\}$.   Then $Q=+1$ if $\hat W$ exits $\Half$ at $\{\hat a + 4\pi k: k \in \Z\}$, and $Q=-1$ if $\hat W$ exits $\Half$ at $\{\hat a + 4\pi (k+1/2): k \in \Z\}$.
For $z\in\DD, a\in\p\DD$, let
\[   \lambda(z,a )  = \lambda_{\DD}(z,a) =  \E^{z,a}[Q ] 
  .\]
For the  simply connected domain $D_A$, we define for $z\in D_A$ and $a\in\p D_A$  
 \begin{equation}\label{lambdaA}
 \lambda_A(z,  a )  = \lambda(f_A(z), f_A(a)),
 \end{equation}
where we recall $f_A: D_A \rightarrow \DD$ with $f_A(w_0) = 0,
f_A'(w_0) > 0$.

\begin{remark} We can also equivalently consider a function $\hat \lambda$ living on the Riemann surface given by the branched two-cover of $\DD$ with branch cut $\beta$. We lift the $h$-process in $\DD$ so that it starts on the top sheet of the two-cover. The observable is the expectation of the random variable giving $+1$ if the $h$-process reaches the boundary on the top sheet and $-1$ if it reaches it on the bottom sheet. Then $\hat{\lambda}$ only changes sign when evaluated at the two different points in the fiber of $z \in \DD$ and so could in physics language be called a ``spinor''. See \cite{ic_spinor} for a similar construction in a related context.  
\end{remark}

Note that symmetry implies that    $\lambda(z;a ) = 0$
for $z \in \alpha$. 
Let $\tau = \inf\{t: W_t \in \alpha\}$.  If $\tau > T$,
then 
  the value of $Q$ is determined:  it is $+1$
if $z$ and $a$ are in the same component of $\Disk \setminus
(\alpha \cup \beta)$ and $-1$ if they are in different components.
Therefore,
\begin{equation}\label{lamb}
  |\lambda(z,a)| = | \E^{z,a}[Q; \tau > T]|
   =
   \mathbf{P}^{z,a}\left( W[0,T] \cap \alpha = \emptyset \right) = \frac{h_{\Disk \setminus \alpha}(z,a)}{h_{\Disk}(z,a)},
   \end{equation}
the last expression being a quotient of Poisson kernels.
\begin{lemma}  \label{sincos}
If $0 \leq \theta < \pi$, then as $\epsilon \downarrow 0$,
\[  \lambda (-\epsilon;a) = 2 \, \epsilon^{1/2} \,
  \sin \theta_a + O(\epsilon). \]
 \end{lemma}

\begin{proof}
The map 
\[   \phi(z) = \frac{2 \sqrt z}{z+1} , \;\;\;\; \phi'(z) = \frac{1 - z}
      {\sqrt z \, (z+1)^2} \]
      is a conformal transformation of $\Disk \setminus [0,1)$ onto
      the upper half plane ${\mathbb H}$ with 
      \[  \phi(e^{2i\theta}) = \frac{1}{\cos \theta}, \;\;\;\;
      | \phi'(e^{2i\theta})| = \frac{\sin \theta }{2\cos^2 \theta} , \;\;\;
   \phi(\ee e^{2i\mu}) = 2 \, \ee^{1/2} \, e^{i\mu} \, [1 + O(\ee)].\]
The scaling rule for the Poisson kernel implies that
\begin{align*}  \pi \, h_{\Disk \setminus [0,1)}(z, e^{i2\theta})
  &= \pi\, | \phi'(e^{i 2 \theta})| \, h_{\mathbb H}( \phi(z),
  \phi( e^{i2\theta}))   \\
  &= \frac{ | \phi'(e^{i 2 \theta})| \, \Im[ \phi(z)]}
        {   [ \phi( e^{i2\theta})  - \Re  \phi(z)]^2 + [\Im \phi(z)]^2 }  . \end{align*}
        In particular,
    \[   2\pi \,  h_{\Disk \setminus [0,1)}(\ee e^{i 2\mu}, e^{i2\theta})
            =  { 2\ee^{1/2} \,\sin\mu \sin \theta}   + O(\ee).\]
            Therefore, by rotational symmetry, with $a=e^{2i\theta_a}$, 
            $$2\pi h_{ \Disk \setminus \alpha}(-\eps,a) = 2\pi h_{\Disk \setminus [0,1)}(\eps e^{i2\theta_a},-1)=2\eps^{1/2}\sin \theta_a + \bigo{\eps}.$$
            Using in \eqref{lamb} the fact that for any $a\in\p\DD, \lambda (-\eps,a)>0$ and that $2\pi h_{\DD}(-\eps,e^{i2\theta_a})=1+\bigo{\eps}$ now concludes the proof. 
  \end{proof}          
\subsection{Asymptotics of $\Lambda_A(0,a)$}

   Our construction will be
 somewhat neater if we start with the observation in
 the next lemma.
 
 \begin{lemma} \label{6789}
There exists $\epsilon > 0$ such that the following holds.
Suppose $D$ is a simply connected domain containing the origin
and $f: \Disk \rightarrow D$ is a conformal transformation with
$f(0) = 0$.  Let $S = \{ x+ iy: |x|,|y| < |f'(0)| \epsilon \}$
be a square centered at the origin and for $0\leq t < 1$, let $\gamma(t) =
f(t)$. Let $\sigma = \inf\{t: \gamma(t) \in \p S\}$.
Then $\gamma[\sigma,1) \cap S = \emptyset $. 
\end{lemma}

\begin{proof}  This can be proved using standard distortion
estimates for conformal maps.
\end{proof}

We now set some notation.  

\begin{itemize}\setlength{\itemsep}{5pt}

\item 
Let $U_n =\{(x,y)\in \Z^2:|x|<n, |y|<n\}$ denote the discrete square
centered at $0$ of side length $2n$ and $  U_n^- = U_n \setminus \{0, \ldots, n\}$ be the slit discrete square. 

\item  Let $V = V_n, V^- = V_n^-$ be the corresponding continuum domains
\[  V_n = \{x+iy: |x|, |y| < n \},\;\;\;V_n^- = V \setminus (0,n].\]

 \item Let $\epsilon$ be a fixed positive number
as in Lemma \ref{6789}. For every $A$, let  $n_A = \lfloor \epsilon r_A  \rfloor
 = \epsilon \, r_A + O(1)$ and
let  $U_A, U_A^-,V_A,V_A^-$ denote
 $U_{n_A}, U_{n_A}^-,V_{n_A},V_{n_A}^-$, respectively.

 \item  Let $\beta^* =(0,w_0] \cup f_A^{-1}(\beta)  $.
The curve $\beta^*$ goes from
 $0$ to $f^{-1}_A(1) \in \p D_A$. 
By Lemma \ref{6789} we can see that after the curve
$\beta^*$ hits $\p V_A$ it does not return to $V_A$.
We will also write $\beta^*$ for the arc $\beta^* \cap
\overline {V_A}$.

 \item For $z,w \in \overline{V^-_A} \setminus
 \beta^* $ we define $Q^{z,w} = -1$ if $\beta^*   $   separates $z$ from $w$   in $V^-_A$ and $+1$ otherwise.

 \item  We use $S_j$ and $B_t$, respectively, to denote random walk and Brownian motion, as well as the $h$-processes defined from them. The probability measure will always make it clear whether we are dealing with the unconditioned or the conditioned process. We use $\Prob^{z,a},
 \E^{z,a}$ to denote probabilities and expectations with
 respect to an $h$-process conditioned to leave the domain
 at $a$.  We will use this notation for both $S_j$ and $B_t$.
 This should not cause confusion.   All $h$-processes in
 this paper will be those given by boundary points, that is, where
 the harmonic function is the Poisson kernel.  We recall that
 \begin{equation}\label{dischrel}
\Prob^{z,a}
\{
(S _0, \ldots S  _k)=(\omega_0, \ldots ,\omega_k)
\}  = \frac{H_A (\omega_k,a)}{H_A (z,a)}\,
\Prob^{z}
\{ (S _0, \ldots S  _k)=(\omega_0, \ldots ,\omega_k) \} ,
\end{equation}
with a similar formula for the Brownian $h$-process. 
\end{itemize}
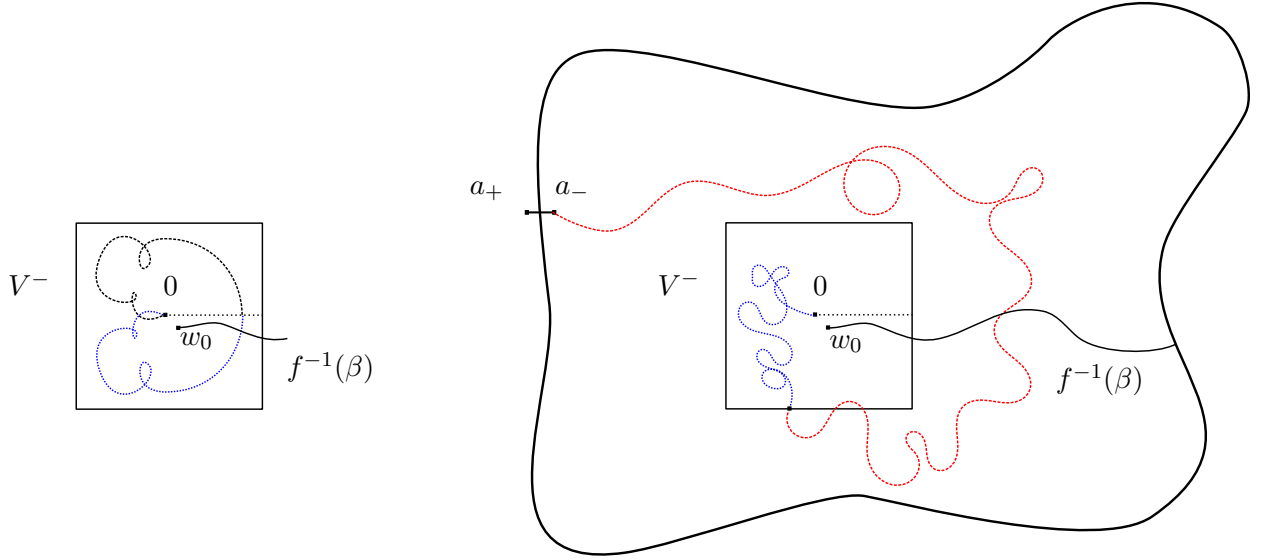
\begin{figure}[t]
  \centering
    \def\svgwidth{\columnwidth}
  
  \begingroup%
  \makeatletter%
  \providecommand\color[2][]{%
    \errmessage{(Inkscape) Color is used for the text in Inkscape, but the package 'color.sty' is not loaded}%
    \renewcommand\color[2][]{}%
  }%
  \providecommand\transparent[1]{%
    \errmessage{(Inkscape) Transparency is used (non-zero) for the text in Inkscape, but the package 'transparent.sty' is not loaded}%
    \renewcommand\transparent[1]{}%
  }%
  \providecommand\rotatebox[2]{#2}%
  \ifx\svgwidth\undefined%
    \setlength{\unitlength}{406.82323456bp}%
    \ifx\svgscale\undefined%
      \relax%
    \else%
      \setlength{\unitlength}{\unitlength * \real{\svgscale}}%
    \fi%
  \else%
    \setlength{\unitlength}{\svgwidth}%
  \fi%
  \global\let\svgwidth\undefined%
  \global\let\svgscale\undefined%
  \makeatother%
  \begin{picture}(1,0.44629737)%
    \put(0,0){\includegraphics[width=\unitlength]{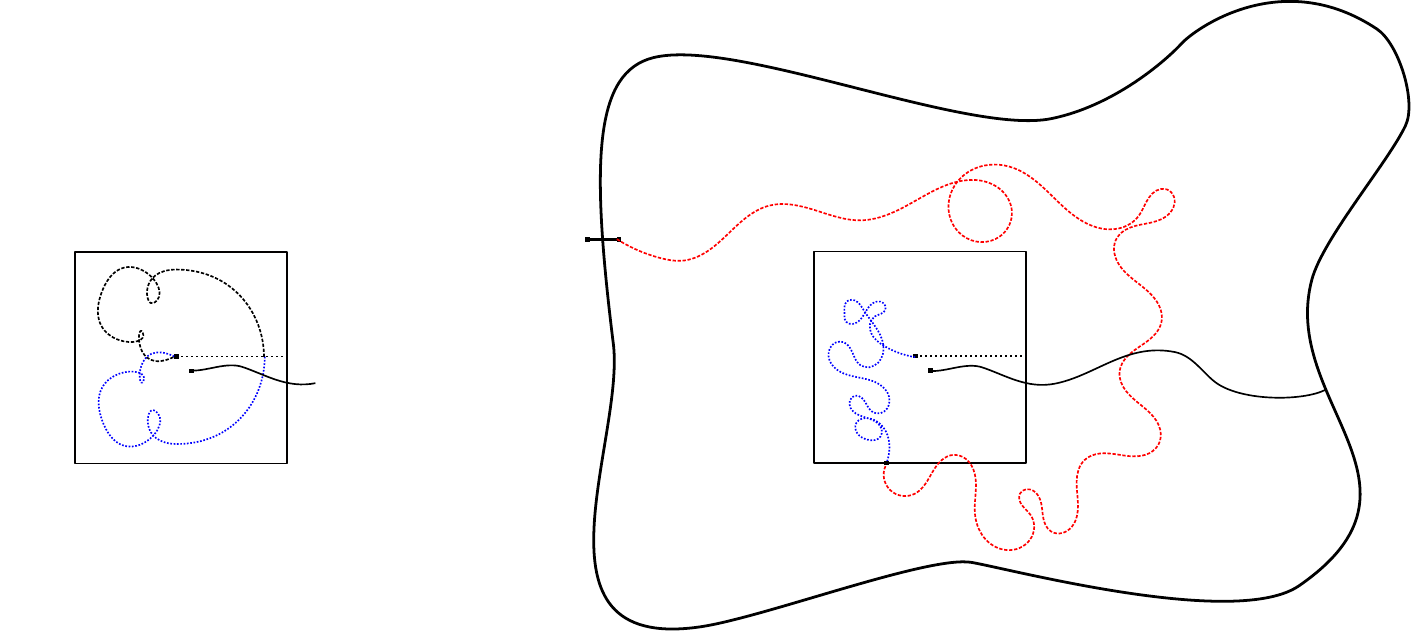}}%
    \put(0.43905066,0.30188198){\color[rgb]{0,0,0}\makebox(0,0)[lt]{\begin{minipage}{0.11688089\unitlength}\raggedright $a_-$\end{minipage}}}%
    \put(0.3702247,0.30188198){\color[rgb]{0,0,0}\makebox(0,0)[lt]{\begin{minipage}{0.11688089\unitlength}\raggedright $a_+$\end{minipage}}}%
    \put(0.64669291,0.22536137){\color[rgb]{0,0,0}\makebox(0,0)[lt]{\begin{minipage}{0.07210643\unitlength}\raggedright $0$\end{minipage}}}%
    \put(0.65972351,0.17765086){\color[rgb]{0,0,0}\makebox(0,0)[lt]{\begin{minipage}{0.07210643\unitlength}\raggedright $w_0$\end{minipage}}}%
    \put(0.84217462,0.13481097){\color[rgb]{0,0,0}\makebox(0,0)[lb]{\smash{$f^{-1}(\beta)$}}}%
    \put(0.52153025,0.21013497){\color[rgb]{0,0,0}\makebox(0,0)[lb]{\smash{$V^{-}$}}}%
    \put(0.12353033,0.22509991){\color[rgb]{0,0,0}\makebox(0,0)[lt]{\begin{minipage}{0.07210643\unitlength}\raggedright $0$\end{minipage}}}%
    \put(0.13679811,0.17998786){\color[rgb]{0,0,0}\makebox(0,0)[lt]{\begin{minipage}{0.07210643\unitlength}\raggedright $w_0$\end{minipage}}}%
    \put(0.22274097,0.14464325){\color[rgb]{0,0,0}\makebox(0,0)[lb]{\smash{$f^{-1}(\beta)$}}}%
    \put(-0.00163231,0.2098735){\color[rgb]{0,0,0}\makebox(0,0)[lb]{\smash{$V^{-}$}}}%
  \end{picture}%
\endgroup%
  
  \caption{Proofs of Lemma~\ref{discretelambda} and Lemma~\ref{lem:continuouslambda}. Left: Paths that hit the horizontal slit before the boundary of the square 
  can be reflected and coupled. Since these paths have winding number of different parity, the total contribution to $\Lambda_A$ of these paths is 
  $0$. Right: Only paths reaching $\p V$ before the horizontal slit contribute to $\Lambda_A$, which can then be written as a sum/integral over the first visited point $w$ of 
  the boundary of the square. The asymptotics of $\Lambda_A(0,a)$ can therefore be found by separately considering the probability of a random walk from $0$ reaching $\p V$ before hitting the horizontal slit (Theorem~\ref{slitsquarehm}), and by using strong approximation to compare the discrete with the continuum Poisson kernels and $\Lambda$ with $\lambda$ on $\p V$ (Theorem~\ref{MGthm} and Proposition~\ref{approxaftersquare}, respectively).}\label{figdecomposition}
\end{figure}

\begin{lemma}\label{discretelambda}
We have
\begin{align*}\Lambda_A(0,a) & =\sum_{w\in\partial\squ_A} H_{\p U^-_A}(0,w) \, \frac{H_A(w,a)}{H_A(0,a)}\,  Q^{0,w}\Lambda_A(w,a).
\end{align*}
\end{lemma}
\begin{proof} 
We write $U, U^-$ for $U_A, U_A^-$.
Let $\sigma = \min\{k \ge 1: S_k\in\N\cup\{0\}\}$.
Since  $S_{\tau\wedge\sigma}\in \{0,1\}$ implies
that  $I=0$,   using \eqref{def:f} 
and the strong Markov property applied at the stopping time $\rho=\inf\{k\geq 0:S_k\in\p U^-\}$ gives
\begin{eqnarray}
\Lambda_A(0,a)& = &  \mathbf{E}^{0,a}[Q(S[0,\rho])\, I\, \Lambda_A(S_\rho,a)] \nonumber\\
& =  & \mathbf{E}^{0,a}[Q(S[0,\rho])  \,  \Lambda_A(S_{\rho},a); S_{\rho}\in\{2, \ldots , n\}] \nonumber\\
& & \hspace{1in} + \mathbf{E}^{0,a}[Q(S[0,\rho])\,\Lambda_A(S_{\rho},a)\,;\,S_{\rho}\in\partial\squ].\label{phidecomp} 
\end{eqnarray}
 
Suppose $\omega$ is a nearest-neighbor path from $0$ to $w\in\{2, \ldots , n\}$, and otherwise staying in $\slit$. Then if $\bar{\omega}$ is the reflection of $\omega$ about the real axis we have $Q(\omega) = - Q(\bar{\omega})$. Indeed, $\omega \oplus \bar{\omega}$ is a loop rooted at $0$ intersecting $\N$ exactly once and it is symmetric about the real axis and so has winding number $1$ about $w_0$. Moreover, by \eqref{dischrel} $\omega$ and $\bar{\omega}$ have the same distribution. This implies that $Q(\omega) = - Q(\bar{\omega})$. Since the (reflected) paths can be coupled after the first visit to $\{2, \ldots , n\}$ we conclude that the first expectation in \eqref{phidecomp} vanishes.

If there exists a nearest neighbor path in $\bar{\squ}$ from $0$ to $w\in\partial\squ$ that does not intersect $\beta^*\cup\{0, \ldots, n\}$ except at its starting and endpoint, then any path from $0$ to $w$ that avoids $\{0, \ldots, n\}$ will intersect $\beta^*$ an even number of times, yielding a value of $Q^{0,w}=1$. Similarly, for all other $w\in\partial\squ$, $Q^{0,w}=-1$. 
We then see that,
\begin{align*}\Lambda_A(0,\ze) & = \mathbf{E}^{0,a}[Q(S[0,\rho])\,\Lambda_A(S_{\rho}, a);S_{\rho}\in\partial\squ]\\
& =\sum_{w\in\partial\squ} \mathbf{P}^{0,a}(S_\rho=w)\,Q^{0,w}\Lambda_A(w,a) \\
& = \sum_{w\in\partial\squ} H_{\p U^-}(0,w) \,\frac{H_A(w,a)}{H_A(0,a)}Q^{0,w}\,\Lambda_A(w,a).
\end{align*}  
\end{proof}
\begin{lemma}\label{lem:continuouslambda}
If $z = z_n = -n^{1/2}$, then 
\begin{align*}
\lambda_{A} (z,a) & = O(n^{-3/4}  )
+ \int_{\p V_A} h_{V^-_A}(z,w) \frac{h_{A}(w,a)}{h_{A}(z,a)}   Q^{z,w}  \lambda_{A}(w,a) \, |dw| \\
& = O(n^{-3/4}  )
+ \sum_{w \in \p U_A} h_{V^-_A}(z,w) \frac{h_{A}(w,a)}{h_{A}(z,a)} Q^{z,w}  \lambda_{A} (w,a)  . 
\end{align*}
\end{lemma}

\begin{proof}  
The proof of the first expression is similar to that of the last lemma.
Note, however, since the winding is measured around $w_0$ it is possible
that when we concatenate a path from $z$ to  the positive
real axis with its reflection about the real axis we could get a
loop whose winding number about $w_0$ is even.  By topology
we can see this can only happen if the path hits $ \eta:= [0,w_0] 
\cup [0,\overline{w_0}]$  before hitting $[0,\infty)$.
By the Beurling estimate, the probability of hitting $\eta
 $ before hitting the positive real axis
is $O(|z|^{-1/2})$.  Also the value of $\lambda $ on
$\eta$  is $O(n^{-1/2})$.  Therefore, this term produces an
error of size $O(n^{-3/4})$
, which yields the first equality of the lemma.

For the second estimate we first note that the Beurling estimate and the covariance rule for the Poisson kernel show that $\forall w \in \p V$
\[
|h_{V^-}(z,w)\, \frac{h_A(w,a)}{h_A(z,a)} \,Q^{z,w} \, \lambda_{A} (w,a)| \le c |h_{V^-}(z,w)|\le c' |z|^{1/2}  n^{-3/2}.  
\]
Let $E$ be the set obtained by removing from $\p V$ its intersection with the six balls of radius $n^{1/2}$ centered at the four corners of $V$, the point at which the slit meets $\p V$, and the point at which $\beta^*$ meets $\p V$. Then by the last estimate 
\begin{multline*}
\int_{\p V} h_{V^-}(z,w) \frac{h_A(w,a)}{h_A(z,a)}   Q^{z,w}  \lambda_{A} (w,a) \, |dw| \\ = \int_{E} h_{V^-}(z,w) \frac{h_A(w,a)}{h_A(z,a)}   Q^{z,w}  \lambda_{A} (w,a) |dw| + O(|z|^{1/2} n^{-1}).
\end{multline*}
Notice that $Q^{z, w}$ is constant on each component of $E$. Derivative estimates for positive harmonic functions show that if $u,v$ are in the same component of $E$ and $|u-v|\le 1$ then 
\begin{align*}
\frac{h_A(u,a)}{h_A(z,a)}&=\frac{h_A(v,a)}{h_A(z,a)}(1+O(n^{-1})).
\end{align*}
Finally, since each point on $E$ is distance at least $n^{1/2}$ from a corner one can map to the unit disk and compare the Poisson kernels there (using, e.g., Schwarz reflection and the distortion theorem to compare the derivatives) to see that
\[
h_{V^-}(z,u)=h_{V^-}(z,v)(1+O(n^{-1/2})).
\]
These estimates imply the lemma.  
\end{proof}

We proceed to show that each of the factors in the expression for $\Lambda_A(0,a)$ in Lemma \ref{discretelambda} is close to its continuum counterpart in the decomposition of Lemma \ref{lem:continuouslambda} and then appeal to Lemma~\ref{lem:continuouslambda} to go back to the continuum function. The estimates naturally split into two parts. The first deals with fine asymptotics close to the tip of the slit in the slit square and the second with what happens between the boundary of the slit square and the boundary of $A$. We state the results here, but postpone the proofs to the two subsequent subsections. We then combine them to prove Theorem~\ref{thm:theorem4-new}.

We define
\[   \bar H_n(0,b) = \frac{H_{\p U_n^{-}}(0,b)}{
  \sum_{z \in \p U_n} H_{\p U_n^{-}}(0,z)}, \;\;\; b \in \p U_n , \]
  This is the conditional probability that an excursion
starting at $0$ in $U_n^- $ exits $\p U_n^-$ at $b$ given that
it exits at a point in $\p U_n$. We also define the analogous
quantity for Brownian motion,
\[ \bar h_n(-1,b) = \frac{h_{  V_n^{-}}(-1,b)}{
 \int_{\p V_n}  h_{V_n^{-}}(-1,w) \, |dw| }, \;\;\; b \in  \p V_n. \]
\begin{theorem}  \label{slitsquarehm} If $b \in \p U_n$,
\[       \bar H_n(0,b) =
 \bar h_n(-1,b)  \, \left[1 + O(n^{-1/20})\right].\]
\end{theorem}

\begin{proof}
See Section~\ref{slitsquarethmsection}. 
\end{proof}

We do not believe that this error term is optimal, but
this suffices for our needs and  is all that we will prove.  We
will also need the following corollary, which in particular implies the sharpness of the Beurling estimate. 
 
\begin{corollary}
There exist $c,u>0$ such that
\[     \sum_{b \in \p U_n}  H_{\p U_n^{-}}(0,b)  =
    c \, n^{-1/2} \, [1 + O(n^{-u})]. \]
\end{corollary}

\begin{proof}   Let
\[  K(n) =  \sum_{b \in \p U_n}  H_{\p U_n^{-}}(0,b)  
 , \]
    \[ \tilde K(n) =  \int_{\p V_n}  h_{V_n^{-}}(-1,w) \, |dw|.\]
  and note that for $r \ge 1$ (we write $U_{rn} = U_{\lfloor rn \rfloor}$)
 \[   K(rn)  = K(n)   \sum_{b \in \p U_n}  \overline H_{n}(0,b)\,
       H_{U_{rn}^-}(b,\p U_{rn}) , \]
   \[  \tilde K(rn) = \tilde K(n) 
        \int_{\p V_n}  \overline h_{n}(-1,w) \, \hm_{V_{rn}^-}(w,
            \p V_{rn})\, |dw|.\]
      Using the previous theorem and  strong approximation,
   we can see that for $1 \leq r \leq 2$,
  \[  \frac{K(rn)}{K(n)} = \frac{\tilde K(rn)}{\tilde K(n)}
       + O(n^{-u}). \]
   By direct calculation using conformal mapping,
   \[       \frac{\tilde K(rn)}{\tilde K(n)} = r^{-1/2} \, +O(n^{-u}).\]
Therefore, allowing a different $u$,
\[    \frac{K(rn)}{K(n)}  =  r^{-1/2} \, +O(n^{-u}),\]
and the result easily follows from this. 
\end{proof}

  Given the corollary we can restate Theorem 
  \ref{slitsquarehm} as: there exists $c >0, u > 0$ such that
\begin{equation}  \label{jan11.1}
   H_{\p U_n^-}(0,b) =c\, 
 h_{V_n^-}(-1,b)  \, \left[1 + O(n^{-u})\right].
 \end{equation}
    We will also need that
\begin{equation}  \label{jan11.2}
\frac{H_A(w,a)}{H_A(0,a)}=
\frac{h_{A}(w,a)}{h_{A}(0,a)} + O(n^{-u}) ,\;\;\;\;
  w \in \p U_n^- , \;\;\; a \in\p A
  \end{equation}
which is a direct consequence of Theorem~\ref{MGthm}.
This handles the first part of the argument of Theorem \ref{thm:theorem4-new}. The second
part is handled in the next proposition.  Note that
for $w \in \p U_A$ 
 the quantities $\Lambda_A(w,a)$
and $\lambda_A(w,a)$ are ``macroscopic'' quantities for which
one would expect Brownian motion and random walk to give
almost the same value.  
\begin{proposition}\label{approxaftersquare} There exist $u>0, c < \infty$ such that if 
  $A\in \mathcal{A}$, $a \in \p A$, and $w\in\partial\squ_A$,
\begin{equation}\label{diffoutside}
|\Lambda_A(w,a)-\lambda_{A}(w,a)|\leq c \, r_A^{-u}.
\end{equation}
\end{proposition}
\begin{proof}
See Section~\ref{approxaftersquaresection}. 
\end{proof}
\begin{proof}[Theorem~\ref{thm:theorem4-new} assuming Theorem~\ref{slitsquarehm} and Proposition~\ref{approxaftersquare}.]
 We will first show \eqref{t41}, that is, 
  there is a constant $0 < c_2 < \infty$ such that 
\begin{equation*}  
\Lambda_A(0,a) =
     c_2 r_A^{-1/2}\, \left[ \sin \theta_a  + O\left(
r_A^{-u} \right) \right].
    \end{equation*} 
    We write $U,U^-,V,V^-,\bar H, \bar h$ for
    $U_A, U^-_A,V_A,V^-_A, \bar H_{n_A}, \bar h_{n_A}$.
Using  Lemma~\ref{discretelambda}, \eqref{jan11.1},  
\eqref{jan11.2}, and \eqref{diffoutside} we see that there is a constant $0 < c_0 < \infty$ such that
\begin{align*}\Lambda_A(0,a) & =\sum_{w\in\partial\squ} H_{\p U^-}(0,w) \frac{H_A(w,a)}{H_A(0,a)} Q^{z,w}\Lambda_A(w,a) \\ 
& = c_0 \sum_{w\in\partial\squ}  h_{V^-}(-1,w) \, \frac{h_{A}(w,a)}{h_{A}(0,a)}\, Q^{z,w}\, \lambda_A(w,a)\left[1+O(r_A^{-u}) \right].
\end{align*}
We note that if $f_A: D_A \rightarrow \DD$ with $f_A(w_0) = 0,
f_A'(w_0) > 0$, then for
 for $z$ with $|z|\leq n_A^{1/2}$, 
\[   f_A(z)  =  r_A^{-1} \, z + O(r_A^{-1}), \]
and hence, by \eqref{lambdaA} and Lemma \ref{sincos},
\begin{equation}
  \label{jan11.6}
       \lambda_A(-n_A^{1/2},a) =  2\, r_A^{-1/2} \, n_A^{1/4}
 \, \sin \theta_a + O(r_A^{-1/2})
   \end{equation}
and there exists some $c>0$ such that for all $z$ with $|z|\leq n_A^{1/2}$, 
\begin{equation}\label{02-08}
       \lambda_A(z,a) \leq  c\, r_A^{-1/2} \, |z|^{1/2}.
 \end{equation}

A simple argument, using conformal mapping say,
shows that
\[
   h_{V^-} (-n_A^{ 1/2},w)    =  n_A^{1/4} \, h_{V^-}(-1,w) 
\, \left[1+O(n^{-1/4})\right] .\]
We now use Lemma \ref{lem:continuouslambda}
and \eqref{jan11.6}
 to conclude that
\begin{eqnarray*}
 n_A^{1/4} \, c_0^{-1} \, \Lambda_A(0,a) & = & O(r_A^{-3/4})
     + \lambda_A(-n_A^{1/2},a) \, \left[1 + O(n^{-u})\right] 
\\
     & = & 
    2\, r_A^{-1/2}n_A^{1/4} \,  \left[ \sin \theta_a + O(n^{-u})\right] .
     \end{eqnarray*}
  Since $n_A =  \epsilon \, r_A +O(1)$, we get \eqref{t41}. 
  We will give
a symmetry argument here 
  to deduce \eqref{t42}  from \eqref{t41}. 
Suppose we replace $J_t$ with 
\[  \tilde J_t  = \left\lfloor \frac{\Theta_t + \pi}{2 \pi} \right\rfloor
   - \left\lfloor \frac{\Theta_0 + \pi}{2 \pi}\right \rfloor. \]  In other words,
 we place the discontinuity of the argument at $ 
 f^{-1}_A((-1,0]) $ rather than at $  f^{-1}_A([0,1))$.  Then
 we can see that if $S_\tau = a $, then
$\tilde Q(\omega[0,\tau]) = \pm Q(\omega[0,\tau])$ 
with the minus sign appearing if and only 
if $\pi/2 \leq \theta_a < \pi$.  
 
 Now given $A$, consider its reflection along
 the line $\{\Re(z) = 1/2\}$, that is, let 
 $\rho(z) = 1 - \overline z$, $A' = \rho(A)$.    
   Let $a' = \rho(a)$ and define $\theta'$
 by $f_{A'}(a') = e^{i2\theta'}$.   Note that $
 \rho(D_{A'}) = D_A$, $f_{A'}(z)
   = -\overline {f_A(\rho(z))}$, $f_A^{-1}([0,1))
     = \rho\left[f_{A'}^{-1}((-1,0])\right],$ and
 \[ f_A(a) = - \overline{f_{A'}(a')}  =
    -\overline{e^{2i\theta'}} =
      \exp \left\{2i\left(\frac{\pi}{ 2} - \theta'\right)\right\}. \]
In other words, $\theta_{a} = \frac \pi 2 -
\theta' \, ({\rm mod} \,\pi)$.  If $\theta_a,\theta' \in [0,\pi)$, then
$\theta_a = \frac \pi 2 - \theta'$ if $0 \leq \theta'  \leq \frac \pi 2$
and $\theta_a = \frac{3\pi}{2} - \theta'$  if $\frac{\pi}{2}
          < \theta' < \pi$, 
       and hence
       \[ \cos \theta_a = \left\{ \begin{array}{ll}
   \sin \theta' & 0 \leq \theta' \leq \pi/2 \\
    -\sin \theta' & \pi/2 < \theta'
      < \pi . \end{array} \right.\]
        From the previous paragraph and \eqref{t41},
         we see that
\[    \Lambda (1,a) = \Lambda_{A'}(0,a')  
= c_2 r_A^{-1/2}\, \left[ \sin  \theta'   + O\left(r_A^{-u}\right)\right] , \;\;\;\;
   0 \leq \theta' \leq  \frac \pi 2 , \]
   \[    \Lambda (1,a) = -\Lambda_{A'}(0,a')  
= c_2 r_A^{-1/2}\, \left[ -\sin  \theta'   + O\left(r_A^{-u}\right) \right] , \;\;\;\;
   \frac \pi 2 < \theta < \pi . \]
 \end{proof}

\subsection{Poisson kernel convergence in the slit square: proof of Theorem \ref{slitsquarehm}}  \label{slitsquarethmsection}

The rate of convergence of the discrete Poisson kernel to
the continuous kernel is very fast in the case of rectangles aligned
with the coordinate axes.

\begin{lemma}  \label{jan10.lemma1}
There exists $ c < \infty$ such that if 
$n/10 \leq m \leq 10 n$, and
\[ A = A(n,m)  = \{j+ik : 1 \leq j \leq n - 1,
  1 \leq k \leq m-1\} , \]\[     R=R(n,m) = \{x+iy: 0 < x < n,
    0 < y < m\}, \] then     
    \[   \left|4\, h_{\p R}(ik, n + i k') - H_{\p A}
       (ik, n+i k')\right|  \leq c \, \frac{\min\{k, m-k\}\min\{k', m-k'\}}{{n^5} }. \]
In particular,
\begin{equation*}\label{relative}
H_{\p A} (ik, n+i k') = 4\, h_{\p R}(ik, n + i k') (1+\bigo{n^{-2}}).
\end{equation*} 
   \end{lemma}

\begin{proof}
We write $ v=ik, w= n + i k'$, and $d= \min\{k, m-k\}, d' = \min\{k', m-k'\}$. 
Using separation of variables (see Section 6 of \cite{benes_rw} or \cite[Chapter 8]{lawler_limic}
for more details), one can find $h_R(1+ ik,
  w)$ exactly as an infinite series
  \[   h_R(1+ik,w) = \frac 2 m \sum_{l=1}^\infty
    \frac{\sinh(l  \pi/m)}
        {\sinh (  l\pi n/m) }\, \sin\left(\frac{lk\pi}{m}
        \right) \, \sin \left(\frac{l k' \pi}{m} \right).\]
Similarly one can find the discrete
Poisson kernel by separation of variable as  a finite Fourier series; more specifically,
\[     H_{A}(1+ ik, w) =
    \frac{2}m\sum_{l = 1}^{m-1} \frac{\sinh(\alpha_l \pi/n)}
        {\sinh (\alpha_l  \pi) }\, \sin\left(\frac{lk\pi}{m}
        \right) \, \sin \left(\frac{l k' \pi}{m} \right) , \]
  where $\alpha _l$ is the solution to
  \[    \cosh\left(\frac{\alpha_l \pi}{n} \right)
     + \cos\left(\frac{l \pi}{m} \right) = 2 . \]
   Note that Lemma 6.1 in \cite{benes_rw} implies that
\begin{equation}  \label{sept19.1}
  \alpha_l = \frac{ln}{m} \, \left[ 1 + O(l^2/n^2) \right]. 
  \end{equation}
One can find $c_1>0$ and $c_2>0$ such that for all $m,n$ in the statement of the lemma, 
$$  \frac{\sinh(l  \pi/m)}
        {\sinh (  l\pi n/m) }\leq c_1e^{-c_2 l} \text{ and }  \frac{\sinh(\alpha_l \pi/n)}
        {\sinh (\alpha_l  \pi) }\leq c_1e^{-c_2 l} .$$

Using
this  and the inequality $|\sin x| \leq |x|$,
we can see that there exists  $c < \infty$ such that
  \[  |h_R(1 + ik,w) - H_A(1 + ik,w) |   \leq  \hspace{1.8in} \]
  \[  c \left[\frac{1}{n^6}  +\frac{d d'}{n^3} \sum_{l \leq c \, \log n}
         l^2  \left| \frac{\sinh(l  \pi/m)}
        {\sinh (  l\pi n/m) } - \frac{\sinh(\alpha_l \pi/n)}
        {\sinh (\alpha_l  \pi) }\right| \right]. \]
Note that the 6 is arbitrary and can be made arbitrarily large by increasing $c$.  Using \eqref{sept19.1} we can see that for $l \leq c \log n$,
   \[   \frac{\sinh(\alpha_l \pi/n)}
        {\sinh (\alpha_l  \pi) } = \frac{\sinh(l  \pi/m)}
        {\sinh (  l\pi n/m) }\, \left[ 1 + O(l^3/n^2) \right], \]
  and hence   the sum is $O(n^{-3})$
giving
$$
      |h_R(1 + ik,w) - H_A(1 + ik,w) |   \leq \frac{c dd'}{n^6} .
$$

This implies that
\[   H_{\p A}(ik, w) = \frac 14 \,  H_A(1 + ik, w) = 
\frac 14 \, h_{R}(1+ik,w) 
      + O(dd'n^{-6}) . \]
We now assume that $k\leq m/2$. We can extend the function $h(z) = h_{R}(z,w)$ to
$\{ x+iy: -n < x < n, -m < y < m \}$ by Schwarz
reflection. (If $k > m/2$, we instead extend the function to $\{ x+iy: -n< x < n, 0 < y < 2m \}$.) Note that on $\{z:  |z - ik| \leq \min\{m,n\}/4 \},$      
$$|h(z) | \leq c \, dd'/n^3.$$
 Since $\min\{m,n\} \asymp n$,
  estimates for derivatives of harmonic functions (see, for instance, Section 2.3 in \cite{Lawler_cip}) then imply that
 \begin{equation}\label{harmderiv} 
   |\p_{x}^kh(z)| \leq c \, dd'/ n^{3+k} 
   \end{equation}
Using the definition of the boundary Poisson kernel (see Section 2.3) and \eqref{harmderiv} in a Taylor expansion of $h$ at $ik$, we see that 
\[       h_R(ik + 1,w)  =   h_{\p R}(ik,w)  +  O(dd'/n^5).\]  
\end{proof} 
Let us abuse notation slightly and write
\[    G_{U_n^-}(0,\zeta) = \frac{1}{4}
 \sum_{|e| = 1} G_{U_n^-}(e,\zeta). \]

\begin{lemma}\label{01-22-14}    For every $n$, there exists $c_n$ such
that  for all $|\zeta| \geq n/2$, 
\[   G_{U_n^-}(0,\zeta)  =
         c_n \, g_{V_n^-}(-1,\zeta)  \, \left[ 1 +  
      O(n^{-1/20})\right] . \]
 Moreover, the constant $c_n$ is uniformly bounded away
 from $0$ and $\infty$.
\end{lemma}

\begin{proof}
 Let $z_0 = - \lfloor n/8 \rfloor$ and $F_n: V_n^- \rightarrow \Disk$ be  the  conformal transformation
 with $F_n(z_0)=0, F_n(0) = 1$.  Note that for all $z\in V_n^-, F_n(z) = F(z/n)$
 where $F = F_1$. 
 It is easy to see that $|F_n(\zeta) -1| = |F(\zeta/n)
  - 1|$  is uniformly bounded
 away from $0$ for $|\zeta| \geq n/2$.  
 
Let $H_n$ be the restriction to $V_n^-$ of the Schwarz-Christoffel transformation from $V_n$ to $n\Disk$, that sends the origin to the origin and $(n,0)$ to $(n,0)$. Then the image of $H_n$ is $n\Disk\setminus [0,n]$. We can see, e.g., from the explicit form of $H_n$ that $H_n(z_0)=-cn(1+\bigo{1/n})$ for some $0<c<1$. Moreover, $H_n(-1)=-H_n'(0)(1+\bigo{1/n})$ and $H'_n(-1)=c'(1+\bigo{1/n})$ for some $c'>0$. 

 We can then write $F_n = G_n\circ H_n$, where $G_n(z)=(1-z_a)(\sqrt{z/n}+\sqrt{n/z})+(1+z_a)i/(1-z_a)(\sqrt{z/n}+\sqrt{n/z})-(1+z_a)i$ ($z_a$ is some real in $[0,1]$ that depends on $H_n(z_0)$ and can be computed explicitly) maps $n\Disk\setminus [0,n]$ to $\Disk$, $H_n(z_0)$ to 0 and $0$ to $1$. Then $G_n'(H_n(-1))= c'' \, n^{-1/2} \, [1 + O(n^{-1/2})]$ for some  $c''$, so that the chain rule implies that $F_n'(-1)=c_0 \, n^{-1/2} \, [1 + O(n^{-1/2})]$ for some constant $c_0$.
 
 Using the explicit form of the Green's function and Poisson kernel in the disk, we
 can see that 
\begin{eqnarray*}
g_{V^-}(-1, \zeta) & =  & g_\Disk(F_n(-1), F_n(\zeta)) \\&=&
       g_\Disk(1 - c_0 \, n^{-1/2}(1 + O(n^{-1/2})) , F(\zeta/n))
         \,  \\
  & = & 2\pi \,c_0 \, n^{-1/2} \,  h_{\Disk}(F(\zeta/n),1) \, 
   [1 + O(n^{-1/2})] \\ &  = & \frac{2 \pi \, c_0 \, [1 - |F(\zeta/n)|^2]}
      {n^{1/2} \,|F(\zeta/n) - 1|^2 } \,   [1 + O(n^{-1/2})] . 
      \end{eqnarray*}
  By equation (40) of  \cite{KL},
we can see that 
 \[      G_{U^-}(0,\zeta) = G_{U^-}(0,z')   
  \frac{  1 -| F(\zeta/n)|^2}
      { |F(\zeta/n) - 1|^2 } \,   [1 + O(n^{-1/20})].\]
      Here we are using the
 uniform  bound   on $|F(\zeta/n)
  - 1|$.   
All one needs now is that $G_{U^-}(0,z_0) \asymp n^{-1/2}$, which follows from Proposition 1.5.9 and (2.40) in \cite{greenbook}.
 \end{proof}

\begin{proof}[Theorem \ref{slitsquarehm}.]
 Let us first consider the case $b = n + ik'$ with $0 < k' <
 n$, 
Let $m = \lfloor 3n/4 \rfloor $ and let $R= R_n$ be the
 rectangle in the top right corner of $U$:
\[   R = \{x +iy: m < x < n, 0 < y < n \}.\]
As an abuse of notation we will also write $R$ for $R \cap
\Z^2$. 
A last-exit decomposition shows that
\[     H_{\p U^-} (0,b) = \sum_{j=1}^{n-1}
    G_{U_n^-}(0,m+ji) \, H_{\p R}(m+ji,b) . \]
  A similar decomposition shows that
 \[       h_{V^-}(-1,b) = \frac{1}{2 \pi} \int_0^n
      g_{V^-}(-1,m+i\zeta) \, h_{\p R}(m+i\zeta, b) \, |d\zeta|. \]
 This latter equality is perhaps better seen by writing
 \[    h_{V^-}(-1,b) = \lim_{\epsilon \downarrow 0}
          \frac 1 {2\pi \epsilon}  g_{V^-}(-1, b - \epsilon)
           =   \lim_{\epsilon \downarrow 0}
          \frac 1{2\pi \epsilon}  g_{V^-}(  b - \epsilon, -1), \]
          and using the strong Markov property. 
Lemma \ref{jan10.lemma1}
shows that
\begin{equation} \label{jan10.1}
 H_{\p R}(m+ji,b) =  \frac{h_{\p R}(m+ji, b)}
             4 + O(d/n^4) , 
             \end{equation}
         where $d = \min\{k', n-k'\}$.  
 For $A\in\Z^2$, we let $\tau_A = \inf\{n\geq 1: S(n)\in A\}$ and write $\tau_x$ when $x$ is just a point. Then, using again Proposition 1.5.9 and (2.40) in \cite{greenbook}, for $1\leq j\leq n-1$,
 \begin{eqnarray*}
 G_{U^-}(0,m+ij)& = & \Prob^0(\tau_{m+ij}<\tau_{\p U^-})G_{U^-}(m+ij,m+ij)\\
 & \leq & cn^{-1/2}G_{U^-}(m+1+ij,m+ij) \leq cn^{-1/2},
 \end{eqnarray*}
so
   \[   \sum_{j = 1}^{n-1} G_{U^-} (0,m+ji) \leq cn^{1/2}.\]
  Therefore, using \eqref{jan10.1},
  \begin{eqnarray*}
      H_{\p U^-}
     (0,b) & = &   \sum_{j=1}^{n-1}
    G_{U_n^-}(0,m+ji) \,\left[
    H_{\p R}(m+ji,b) -  \frac{1}{4} h_{\p R}(m+ji,b) \right]
      \\
  &   & \hspace{1in} + \frac 14 \sum_{j=1}^{n-1}
    G_{U_n^-}(0,m+ji) \, h_{\p R}(m+ji,b) 
  \\
  & =  &  O(d/n^{7/2}) +  \frac 14 \sum_{j=1}^{n-1}
    G_{U_n^-}(0,m+ji) \, h_{\p R}(m+ji,b) 
     . 
     \end{eqnarray*}
  By Lemma \ref{01-22-14}, we can write
  \[  H_{\p U^-}
     (0,b) =   O(d/n^{7/2}) + [1 + O(n^{-1/20})]
      \sum_{j=1}^{n-1}
   c_n \,  g_{V_n^-}(-1,m+ji) \, h_{\p R}(m+ji,b) 
, \]
where this $c_n$ is $1/4$ times the value
in that lemma. 
  The sum above is greater than a constant times
  $d/n^{3/2}$. Indeed, the probability that a Brownian motion from $-1$ reaches a point in (say) the middle half of the left side of $R$ before exiting $V_n$ is at least $cn^{-1/2}$. Moreover, the function $h_{\partial{R}}(\cdot, b)$ at such points is at least $c d/n$. (Recall that $d = \min\{k', n-k'\}$ and $b=n+ik'$.) Hence we can write this as
   \[  H_{\p U^-}
     (0,b) = [1 + O(n^{-1/20})]
      \sum_{j=1}^{n-1}
   c_n \,  g_{V_n^-}(-1,m+ji) \, h_{\p R}(m+ji,b) 
       . \]
   Routine estimates allow us to approximate the
   sum by an integral,
\begin{eqnarray*}
 H_{\p U^-}
     (0,b) & = &  [1 + O(n^{-1/20})]
       \int_0^{n} 
    c_n \, g_{V_n^-}(-1,\zeta) \, h_{\p R}(\zeta,b)  \, |d\zeta|
    \\
    &    =  &  2\pi \, c_n \,   h_{V^-}(-1,b) \, [1 + O(n^{-1/20})].
\end{eqnarray*}

We assumed that $b = n +i k'$ with $k ' > 0$.  There are four
other cases.  For example, if $b = k + ni$ with $k > 0$ we
replace the rectangle $R$ with the rectangle
\[     \{x+iy:  - n < x < n, m < y < n \} . \]
We do the other three cases similarly.  In all cases
we choose $z_0 = - \lfloor n/8 \rfloor$.   The same
argument gives us
\[  H_{\p U^-}
     (0,b) =    2\pi \, c_n \,   h_{V^-}(-1,b) \, [1 + O(n^{-1/20})], \]
     with the same value of $c_n$. From this we conclude that
     \[
     \frac{H_{\p U^-}(0,b)}{\sum_{y \in \p U} H_{\p U^-}( 0, y)} = \frac{h_{V^-}(-1,b)}{\int_{\p V}h_{V^-}(-1,y) |dy|} \left[1+O(n^{-1/20}) \right].
     \]       
\end{proof}

\subsection{From the slit square to $\partial A$: proof of Proposition  \ref{approxaftersquare}}\label{approxaftersquaresection}

To prove Proposition \ref{approxaftersquare}, we need a version of Theorem \ref{kmt} for the $h$-processes. We will however use slightly different stopping times. We will write $f$ for $f_A$ and for $u>0$ consider Brownian and random walk paths $B$ and $S$ with measure either $\Prob^z$ or $\Prob^{z,a}$ (for the unconditioned and conditioned paths, respectively). We let
$$\tau_u 
  = \inf\{t\geq 0: d(f(S_{2t}),\partial \Disk)\leq n^{-u}\} , \;
   T_u = \inf\{t\geq 0:d(f(B_t),\partial \Disk)\leq n^{-u}\}, $$
  \[  \sigma_u = \tau_{u} \wedge T_{u}.\]
  Let $\tau, T$ be the times that the paths reach $\p \Disk$ (of course, under the measure $\Prob^{z,a}$, this is the time at which they reach $a$).

Intuitively, knowing that $S$ and $B$ will exit at the same point $a$ should only make it easier under the measure $\Prob^{z,a}$ than under $\Prob^z$ to find a coupling that ensures that $B$ and $S$ are close with high probability. Indeed, this is the case.
We will prove the following result which does not give the optimal bounds.
\begin{theorem}\label{kmthproc}
There exist $u, u' >0, c < \infty$   such that if $n \leq \dist(z, \p A)
 \leq n+1$,   $\ze\in\partial A$, and  $z \in A$ with $  |z| \leq 3n/4 $  then one can construct a probability space containing 
two processes $B$ and $S$ such that the probability of the event that
\[
\sup_{0\leq t \leq \sigma_u }|S_{2t}-B_t|  \geq c\log n, 
\]
and
\[
 \diam\left(f \circ S[\sigma_u, \tau]\right) +  \diam\left(
f\circ B[\sigma_u,T] \right) \geq c \, n^{-u'}
\]
is bounded above by $cn^{-u'}$.
Here $B$ and $S$ have the distribution of a Brownian, respectively random walk $h$-process started at $z$ and conditioned to leave $D_A$ at $a$.
\end{theorem}
\begin{proof} We use the KMT coupling of Theorem \ref{kmt} to put the
unconditioned paths $B$ and $S$ on a probability space in such a way that if $\mathcal{K} = \{\sup_{0\leq t \leq \sigma_u }|S_{2t}-B_t|  \leq c\log n\}$, we have
$$\Prob^z(\mathcal{K}^c)\leq cn^{-2}.$$ 
To obtain the corresponding result for the $h$-processes, we note that the fact that there is a point $v$ with $|v|=1-n^{-u}$ and $d(v,f(B(\sigma_u))\leq c\log n$, together with the distortion theorem, implies that there exists a constant $c$ such that $d(f(B(\sigma_u)),\p\Disk)\in [c^{-1}n^{-u},cn^{-u}]$, which, using the explicit form of the Poisson kernel in the unit disk in \eqref{poissondisk}, implies that $h(B(\sigma_u),a)\in [c^{-1}n^{-u},cn^{u}]$. Similarly, using Theorem~4.1 in \cite{BJK} and (40) in \cite{KL}, we see that, with a possibly different constant $c, H(S(\sigma_u),a)\in [c^{-1}n^{-u},cn^{u}]$. Moreover, by Harnack estimates $h(z,a) \asymp 1$ and $H(z,a)/H(0,a)  \asymp 1$ for $z$ with $|z| \le 3n/4$. The fact that the measure for the Brownian $h$-process is obtained by weighing the Brownian paths by a factor of $h(B(\sigma_u),a)/h(z,a)$ and that the measure for the random walk $h$-process is obtained by weighing the random walk paths by a factor of $H(S(\sigma_u),a)/H(z,a)$ now implies that
$$\Prob^{z,a}(\mathcal{K}^c)\leq cn^{-2+u}.$$ 

It remains to show that there is $u' > 0$ such that
\[\Prob^{z,a}(\diam\left(f \circ S[\sigma_u, \tau]\right)  + \diam\left(
 f \circ B[\sigma_u,T]\right) \geq c \, n^{-u'})\leq cn^{-u'}.\]
In order to split this into separate estimates for $S$ and $B$, we have to deal with the technical issue that the joint process $(S,B)$ doesn't satisfy the strong Markov property in the coupling. To get around this, one can use a standard tool (see, for instance, the proof of Theorem 3.1 in \cite{BJK}) which consists in introducing stopping times $\tau_u$ for $S$ and $T_u$ for $B$ such that 
\begin{equation*}\label{02-23b}
\max\{\tau_u, T_u\}\leq \sigma_u
\end{equation*}
and
\begin{equation}\label{02-23}
f(S(\tau_u))\in A_{c,u},
\end{equation} 
where $A_{c,u} = \{z\in\Disk: |z|\in[1-c^{-1}n^{-u},1-cn^{-u}]$ and, similarly, $|f(B(T_u))|\in A_{c,u}$, and applying the strong Markov property at those times.
Let us just consider the $h$-process $S$, as conformal invariance makes the estimate for the $h$-process $B$ considerably simpler. 
 Write $r=n^{-u}$. For simplicity we assume without loss of generality that $f_A(a) = 1$. 
 
 Using the expression for the Poisson kernel in the disk in \eqref{poissondisk} we see that for $R_1\geq 1, h((1-r)e^{iR_1r},1)=\bigo{R_1^{-1}}$. 
We can use this and \eqref{MGthm.eq} to
see that if $w$ is a lattice point within one unit of $f^{-1}((1-c'r)e^{iR_1r})$ with $c'\in[c^{-1},c]$ and $c$ as in \eqref{02-23}, and if $|z|\leq 3n/4$, 
$$\frac{H(w,a)}{H(z,a)}=\bigo{R_1^{-1}}.$$ 
Therefore,
\begin{align*}\Prob(F) & :=\Prob^{z,a}(|\Arg(f(S(\tau_u)))|\geq R_1) \\ & = \bigo{R_1^{-1}}\Prob^{z}(|\Arg(f(S(\tau_u)))|\geq R_1) \\ &=\bigo{R_1^{-1}}.
\end{align*}
Now let $R_1=r^{-1/4}$ and $R_2 = r^{-3/4}$ and note that $r^{-1}\geq R_2\geq R_1\geq 1$.
\begin{align*}
\Prob^{z,a}(\diam\left(f \circ S[\tau_u, \tau]\right)\geq R_2r) & \leq  \sup_{v,w}\frac{H(v,a)}{H(w,a)} \cdot \Prob^{z}(F^c;\diam\left(f \circ S[\tau_u, \tau]\right)\geq R_2r) \\ 
& \quad \quad \quad +\bigo{R_1^{-1}}\\
& \leq  C\frac{1}{R_2}\frac{R_1}{R_2r}+\bigo{R_1^{-1}} = \bigo{R_1^{-1}},
\end{align*}
where the sup is over $w\in f^{-1}(A_{c,u}\cap D(1,R_1 r))$  and $v\not\in f^{-1}(D(1,R_2r))$ and we used in the first inequality (40) and (41) of \cite{KL} and in the second inequality Proposition 3.1 in \cite{BJK}, letting a point in $A_{c,u}$ play the role of the origin in that Proposition. Noting that we can let $u'=u^{1/4}$ concludes the more difficult part of the proof.
\end{proof}

\begin{proof} [Proposition \ref{approxaftersquare}.]    We choose constants $c,u,u'$ so that the conclusion of Theorem~\ref{kmthproc} holds: We couple the $h$-processes $B^{\ze}$ and $S^{\ze}$, started at a point $z\in\p U_A$ using the coupling of that theorem and let 
$\mathcal{K}$ be the event that $$\sup_{0\leq t \leq \sigma_u }|S^{\ze}_{2t}-B^{\ze}_t|  \leq c\log n,$$
and
 $$\diam\left(\psi \circ S^a[\sigma_u, \tau]\right)  + \diam\left(
 \psi \circ B^a[\sigma_u,T]\right) \leq c \, n^{-u'},$$
 so that 
 \begin{equation}\label{hkmt}
 P(\mathcal{K}^c)\leq cn^{-u'}.
 \end{equation}
  We define
\[     \xi_b = \inf\{t: |B_t| \leq n^{1/2}+bc\log n\}, \;\;\;\;
        \zeta_b  = \inf\{t: |S_{2t}| \leq n^{1/2}+bc\log n\},\]
        where $c$ is the same constant as in $\mathcal{K}$ and write $\xi$ for $\xi_0$ and $\zeta$ for $\zeta_0$.
        Let $Q^B = Q[B^a[0,T]), Q^S = Q[S^a[0,\tau]].$
Note that $Q^B = Q^S\, I_a$  provided that $\xi > T, \zeta > \tau$,  and $\mathcal{K}$
 holds.
Therefore, if $z \in \p U$, the fact that $|Q^B - Q^S\, I_a|\leq 2$ implies that
\[   \left|\E^{z,a}\left[Q^B - Q^S \, I_a\right]\right|
   \leq \left|\E^{z,a}\left[ Q^B; \xi <  T\right] \right| 
   + \left|\E^{z,a}\left[Q^S \,I_a; \zeta < \tau\right] \right|\hspace{1in}\]
   \[  \hspace{1in}      
        + 2[\Prob^{z,a}(\xi < T; \tau < \zeta; \mathcal{K}) + \Prob^{z,a}(\zeta < \tau; T < \xi; \mathcal{K})
         + \Prob(\mathcal{K}^c)].\]
         
          Since we know from \eqref{02-08} that $|\lambda_A(z,a)| \leq c \, n^{-1/4}$ for
  $|z| \leq n^{-1/2}$, we can use the strong Markov property
  to see that
  \begin{equation}\label{part1}
  \left |\E^{z,a}\left[ Q^B; \xi <  T\right] \right|
   \leq \left |\E^{z,a}\left[ Q^B \mid \xi <  T\right] \right| 
    \leq cn^{-1/4}, 
    \end{equation}
    and similarly,
      \begin{equation}\label{part2}
    \left|\E^{z,a}\left[Q^S \,I_a; \zeta < \tau\right] \right|
    \leq cn^{-1/4}.
    \end{equation}
    
 If we let $\sigma=\inf\{t\geq \zeta_1: |S_{2t}|\geq 2n^{1/2}\}$, we see that 
 \begin{equation}\label{part3}
 \Prob^{z,a}(\xi<T; \zeta>\tau; \mathcal{K}) \leq \Prob^{z,a}(\zeta_1<\tau <\zeta) \leq c\log n/n^{1/2},
 \end{equation}
by the strong Markov property and the planar gambler's ruin estimate
  (the gambler's
  ruin estimate is for simple random walk, but this close to the origin
  the $h$-process is mutually absolutely continuous with respect to
  the simple walk.)  
 
  We can show in the same way that  
  \begin{equation}\label{part4}
  \Prob^{z,a}(\zeta < \tau; T < \xi; \mathcal{K})\leq c\log n/n^{1/2}
 \end{equation}
 Combining \eqref{hkmt}-\eqref{part4} completes the proof. \end{proof}

\bibliographystyle{plain}       
\bibliography{LERWEdge}   

\end{document}